\documentclass[11pt]{amsart}

\usepackage{amsmath, amsthm, amssymb}
\usepackage{mathabx}
\usepackage[hidelinks]{hyperref}
\usepackage{mathrsfs}
\usepackage{xcolor, changepage} 
\usepackage{graphicx}
\usepackage{titletoc}
\usepackage{enumerate}
\usepackage{accents}

\usepackage{caption}
\usepackage{subcaption}

\usepackage{geometry}
\newtheorem{theorem}{Theorem}[section]
\newtheorem{lemma}[theorem]{Lemma}
\newtheorem{cor}[theorem]{Corollary}
\newtheorem{prop}[theorem]{Proposition}
\newtheorem{conj}[theorem]{Conjecture}

\newtheorem{remark}[theorem]{Remark}

\newtheorem{example}[theorem]{Example}
\newtheorem{assumption}[theorem]{Assumption}
\newtheorem{definition}[theorem]{Definition}
\numberwithin{equation}{section}

\def\II{{\rm II}}
\def\tr{{\rm tr}}

\def\R{\mathbf{R}}
\def\B{\mathbf{B}}
\def\S{\mathbf{S}}
\def\H{\mathbf{H}}

\def\Bbb{\mathbb{B}}

\def\zerob{\mathbf{0}}

\def\Ucal{\mathcal{U}}

\def\Hcal{\mathcal{H}}
\def\Acal{\mathcal{A}}
\def\Ccal{\mathcal{C}}

\def\Bcal{\mathcal{B}}
\def\Scal{\mathcal{S}}
\def\Ecal{\mathcal{E}}
\def\Tcal{\mathcal{T}}

\def\Bscr{\mathscr{B}}

\def\xb{{\boldsymbol{x}}}

\def\ed{{\rm d}}

\def\area{{\rm area}}
\def\length{{\rm length}}

\def\vol{{\rm vol}}

\def\tr{{\rm tr}}
\def\div{{\rm div}}
\def\sec{{\rm sec}}

\def\dist{{\rm dist}}

\def\loc{{\rm loc}}

\def\cmc{\mu}

\def\<{\langle}
\def\>{\rangle}

\def\hemiN{\S^n_+}
\def\hemiT{\S^3_+}

\def\setdiff{\backslash}

\title{Rigidity of 3D spherical caps via $\mu$-bubbles}

\author{Yuhao Hu}
\address{Key Laboratory of Pure and Applied Mathematics, 
School of Mathematical Sciences, Peking University, Beijing, 100871, P. R. China
}
\email{yuhao.hu@math.pku.edu.cn}

\author{Peng Liu}
\address{Key Laboratory of Pure and Applied Mathematics, 
School of Mathematical Sciences, Peking University, Beijing, 100871, P. R. China
}
\email{1801110011@pku.edu.cn}

\author{Yuguang Shi}
\address{Key Laboratory of Pure and Applied Mathematics, 
School of Mathematical Sciences, Peking University, Beijing, 100871, P. R. China
}
\email{ygshi@math.pku.edu.cn}

\begin{document}
\maketitle

\begin{abstract}
	By using Gromov's $\mu$-bubble technique, we show that the
	$3$-dimensional spherical caps are rigid under perturbations that 
	do not reduce the metric, the scalar curvature, and the mean curvature along its boundary.
	Several generalizations of this result will be discussed. 
\end{abstract}

%Intro

\section{Introduction}

In recent decades, a lot of progress has been made toward understanding 
the scalar curvature of a Riemannian manifold (see \cite{Gromov21Four} and its references). A particular medium for gaining such understanding 
is to answer whether one can perturb the metric of a `model space' in certain ways without reducing its scalar curvature. This viewpoint was famously represented by the positive mass theorem
 and its various generalizations and analogues. 
One analogue, which 
motivated the current work, is the following conjecture proposed by Min-Oo around 1995  (cf. \cite[Theorem 4]{MinOo95}). 

\begin{conj}\label{minooconj} {\rm (Min-Oo)}
Suppose that $g$ is a smooth Riemannian metric on the (topological) hemisphere $\hemiN$ $(n\ge 3)$ with the properties:
\begin{enumerate}[\qquad\rm (1)]
    \item the scalar curvature $R_g$ satisfies $R_g \ge n(n-1)$ on $\hemiN$;  \label{MinOo1}
    \item the boundary $\partial \hemiN$ is totally geodesic with respect to $g$; \label{MinOo2}
    \item the induced metric on $\partial\hemiN$ agrees with the standard metric on $\S^{n-1}$.\label{MinOo3}
\end{enumerate}
Then $g$ is isometric to the standard metric on $\hemiN$.
\end{conj}

Unlike its counterparts%
\footnote{See \cite[Corollary 2]{SY79Manu}, \cite[Theorem A]{GL83}, \cite{MinOo89} and \cite{ACG08}.} 
modeled on $\R^n$ and $\H^n$,
Min-Oo's conjecture
turned out to admit counterexamples (see \cite{BMN11}). Yet, its statement remains interesting, especially when it is compared with the following theorem of Llaull (see \cite[Theorem A]{Llarull98}).

\begin{theorem}\label{LlarullThm} {\rm (Llarull)}
	Let $(\S^n,\hat g)$ be the standard $n$-sphere $(n\ge 3)$. Suppose that $g$ is another Riemannian
	metric on $\S^n$ satisfying $g\ge \hat g$ and $R_g\ge R_{\hat g}$. Then $g = \hat g$.
\end{theorem}

A side-by-side view of Min-Oo's conjecture and Llarull's theorem suggests the following.

\begin{conj}\label{minooconj2}
Let $(\S^n_+,\hat g)$ be the standard $n$-dimensional hemisphere. 
Then Conjecture~\ref{minooconj} holds under the
additional assumption:
   $g\geq \hat g$.
\end{conj}

Our first result in this article is that
Conjecture \ref{minooconj2} holds when $n = 3$; here is a more precise statement
(cf. Corollary~\ref{bdryMetricCor} below):

\begin{theorem}\label{capRigThm0}
	Let $(\S^{3}_{+}, \hat g)$ be the standard $3$-dimensional hemisphere. Suppose that $g$ is another Riemannian
	metric on $\S^{3}_{+}$ with the properties: 
	\begin{enumerate}[\qquad\rm(1)]
	  \item $g\ge \hat g$ and $R_g\ge R_{\hat g}$  on  $\S^{3}_{+}$;
    	\item the mean curvature%
		\footnote{Given a domain $\Omega$ in a Riemannian manifold, 
		unless we specify otherwise,
		we shall adopt the (sign) convention for the mean curvature of $\partial\Omega$
		to be  $H = \tr(\nabla \nu)$, where $\nu$ is the \emph{outward} unit normal 
		along $\partial\Omega$.
		Under this convention, the mean curvature of the boundary of the unit ball in $\R^n$ is $n-1$.} 
		$H_g$ on  $\partial \S^{3}_{+}$ satisfies $H_g\ge 0$;
	 \item the induced metrics on $\partial\hemiT$ satisfy $g_{\partial\S^{3}_{+}}=\hat{g}_{\partial\S^{3}_{+}}$. \label{hemiRigAssu_metric}
	\end{enumerate}
	Then $g = \hat g$.
\end{theorem}

As we will see below, Theorem~\ref{capRigThm0} admits a somewhat direct proof. 
With more technical work, we can generalize it in the following aspects: (i) the assumption (\ref{hemiRigAssu_metric}) in Theorem~\ref{capRigThm0} will be removed; and 
(ii) the model space will not need to be the standard hemisphere---it can be a `spherical cap' or,  more generally, a geodesic ball inside a space form. To make these points explicit, we now state our main result (cf. Theorem~\ref{capRigThm1_conf} below). 

\begin{theorem}\label{capRigThm1}
For any suitable constants $\kappa, \cmc$, let $(\B_{\kappa,\cmc}, \hat g_\kappa)$ 
be a geodesic ball in the $3$-dimensional space form with sectional curvature $\kappa$ such that $\partial\B_{\kappa,\cmc}$ has mean curvature $\cmc$.  
Suppose that $g$ is another Riemannian metric on $\B_{\kappa,\cmc}$ satisfying
$$
	g\ge \hat g_{\kappa}, \quad R_g \ge 6\kappa \mbox{ on } \B_{\kappa,\cmc}\qquad \mbox{and}\qquad H_g \ge \cmc \mbox{ on } \partial \B_{\kappa,\cmc}.
$$
Then $g = \hat g_\kappa$.
\end{theorem}

In Gromov's first preprint of \cite{Gromov19Mean}, a (more general) version of Theorem~\ref{capRigThm1}
 was stated as a `non-existence' result (see \cite[Theorem 1]{Gromov18MeanPp}); 
 an outline of proof was sketched, which relied on a ``generalized Llarull's theorem". 
Following Gromov's main idea, we present a detailed and purely variational proof of Theorem~\ref{capRigThm1}; this theorem also confirms, in the case of $n = 3$, a rigidity statement mentioned in \cite[Remark (d)]{Gromov18MeanPp} without proof.

A simple modification of the proof of Theorem~\ref{capRigThm1} yields the following (cf. Theorem~\ref{capRigThm2}).
 
 \begin{theorem}\label{capRigThm2_Intro}
 Let $(\S^3\setdiff\{O,O'\},\hat g)$ be the standard $3$-sphere with a pair of antipodal points removed, and let $h\ge 1$ be a smooth function on $\S^3\setdiff\{O,O'\}$.
 Suppose that $g$ is another Riemannian metric on $\S^3\setdiff\{O,O'\}$ satisfying
\[
 	 g\ge h^4 \hat g\quad \mbox{and}\quad R_g\ge h^{-2}  R_{\hat g}.
\]
 Then $h\equiv 1$, and $g = \hat g$.
\end{theorem}

When $h\equiv 1$, Theorem~\ref{capRigThm2_Intro} is a special case of 
Gromov's theorem of ``extremality of doubly punctured spheres"
(cf. \cite[Sections 5.5 and 5.7]{Gromov21Four}), and it implies Theorem~\ref{LlarullThm}
in the case of $n = 3$.
We also remark that Theorem~\ref{capRigThm2_Intro}
 would fail without the assumption $h\ge 1$
(see Remark~\ref{confRmk} below). 
We tend to believe that the conclusion of Theorem~\ref{capRigThm2_Intro} 
still holds when the condition
$g\geq h^{4}\hat{g}$ is replaced by $g\geq h^{2}\hat{g}$;  a condition such as  $\inf h>0$ would still be needed, otherwise, the  metric in Remark~\ref{confRmk} would  serve as a counterexample.   

Before sketching our technical ingredients, let us remind the reader that since the early 1980s, two different approaches---variational and spinorial---have been developed for studying the scalar curvature. Yet, for more than two decades, extensions of Llarull's rigidity theorem, 
like Llarull's original proof, had
 been mainly carried out from the spinorial approach.%
 \footnote{See, for example, \cite{GS02}, \cite{Herzlich05}, \cite{Listing10}, \cite{CZ21} (especially Theorem 1.15, Corollary 1.17) and \cite{Lott21}.}
It is relatively recent that variational methods have also become available for proving results of Llarull type.%
\footnote{To our best knowledge, 
a purely variational proof of Llarull's original theorem remains to be found.} 
A key in this new development, which is also a main tool for the current paper, is Gromov's 
\emph{$\mu$-bubble} technique \cite[Section 5]{Gromov21Four}.

Roughly speaking, given a function $\mu$ on a Riemannian manifold $(M^n,g)$, a
$\mu$-bubble is a minimizer (and a critical point) of the functional
\begin{equation}\label{functional_Intro}
	\Omega\mapsto\vol_{n-1}(\partial\Omega) - \int_{\Omega}\mu
\end{equation}
defined for suitable subsets $\Omega\subset M$; given a $\mu$-bubble, 
useful geometric information can be extracted from its first and second variation formulae.
In order to guarantee that 
a non-degenerate $\mu$-bubble exists, $(M,g)$ is
often assumed to be a \emph{Riemannian band}%
\footnote{See Section~\ref{Sec_mubb_defs} below for definition, and see
\cite{Gromov18GAFA} and \cite{Rade21} for related discussion.}%
, and $\mu$ is often required to satisfy a \emph{barrier condition} (see \eqref{barrierDef} below), which prevents minimizing sequences from collapsing either to a point or into $\partial M$.

In some cases, even without the assumption of either a Riemannian band or a barrier condition, a $\mu$-bubble may still be 
found by direct observation of the functional~\eqref{functional_Intro}. 
This is the case with our proof of Theorem~\ref{capRigThm0}---In fact,  if we modify \eqref{functional_Intro} by considering the new functional
 \begin{equation}\label{functionalMod_Intro}
 	\Omega\mapsto \vol_{n-1}(\partial\Omega) + \int_{\hemiT\setdiff\Omega}\mu,
\end{equation}
the variational properties remain unchanged; in our situation, the new functionals associated to $g$ and $\hat g$ admit an inequality, which becomes an equality
when $\Omega = \hemiT$, and then direct comparison shows that
$\hemiT$ is a $\mu$-bubble (see the proof of Corollary \ref{bdryMetricCor}).
We note that this argument crucially relies on the assumption~(\ref{hemiRigAssu_metric})
in Theorem~\ref{capRigThm0}.

Now let us continue to take Theorem~\ref{capRigThm0} as an example 
to explain how to obtain rigidity results from having an `initial' $\mu$-bubble $\Omega$.
Although $\Omega$ need not be $\S^3_+$, we do, for a technical reason, 
require that $\partial\Omega$ has a connected component $\Sigma_0$ whose projection onto $\S^2$ has \emph{nonzero degree} (see \eqref{projPhiDef})---for simplicity, let us call such a $\Sigma_0$
a `good component'. 
By using the second variation and the Gauss--Bonnet formulae, we show that, under certain extra assumptions, $\Sigma_0$
must be a $2$-sphere parallel (w.r.t. $\hat g$) to the equator $\partial\hemiT$; furthermore, along $\Sigma_0$
the ambient metric $g$ must agree with $\hat g$ (Proposition~\ref{levelsetProp}). 
This obtained, a standard foliation lemma
(Lemma~\ref{foliationLemma}) and minimality of $\Omega$ imply that 
$g$ must agree with $\hat g$ in a \emph{neighborhood} of $\Sigma_0$ (Lemma~\ref{localRigLemma}).
Finally, with an `open-closed' argument and standard facts in geometric measure theory, we show
that such a neighborhood can be extended to the whole manifold, thus completing the proof
(Proposition~\ref{rigidityProp}).

In the more general setting of Theorem~\ref{capRigThm1}, the existence of an `initial' $\mu$-bubble
becomes less direct to prove. For simplicity, let us still assume that the model space is the standard hemisphere.
Although $(\hemiT,g)$ is \emph{not} a Riemannian band, we may consider creating one from it by removing a small geodesic ball centered at the north pole $O\in (\hemiT,\hat g)$, but an immediate 
problem arises:
{the natural choice $\mu = \hat H$ (see \eqref{HhatDef}), 
which corresponds to the mean curvature 
of the geodesic spheres centered at $O$ with respect to $\hat g$, may not satisfy the barrier 
condition.} 

To address this problem, we construct a sequence of perturbations $\mu_\epsilon$ (see \eqref{muepsilon}; also cf. \cite[Section 3]{Zhu21})
of $\hat H$  that \emph{do} satisfy the barrier conditions on a corresponding sequence 
of Riemannian bands
$M_{\epsilon}\subset \hemiT$. In particular, in each $M_\epsilon$ there exists a 
$\mu_\epsilon$-bubble
$\Omega_\epsilon$ (Lemma~\ref{MBepsilonExistenceLemma}). By construction, $\mu_\epsilon$ tends to $\hat H$, and $M_\epsilon$
tends to $\hemiT$, as $\epsilon$ approaches $0$. However, two new questions arise: 
\begin{enumerate}[\qquad(a)]
	\item \emph{As $\epsilon$ tends to $0$, do the
$\Omega_\epsilon$ subconverge to an $\hat H$-bubble $\Omega$ in $(\hemiT, g)$? } 
	\item \emph{If so, does $\partial\Omega$ possess a component whose projection to $\S^2$ has nonzero
degree?}
\end{enumerate}
To put these in a slightly different way, regarding (a), we worry that $\Omega_\epsilon$ may become degenerate
in the limit;  regarding (b), we worry that the `good components' of $\partial\Omega_\epsilon$
may either approach the north pole $O$ and thus lose the `degree' property, or `meet and cancel' each other so that none of them
is actually preserved in the limit. 

In Sections~\ref{sec_noCross} and \ref{limitingSec}, 
we answer both questions (a) and (b) in the affirmative.
A key step is to argue that each $\partial\Omega_\epsilon$ not only possesses a `good component'
$\Sigma^\epsilon_0$, but such a component must be disjoint from a fixed neighborhood of $O\in \hemiT$ provided that $\epsilon$ is small
(Proposition~\ref{noCrossing}), which is, again, enforced by the Gauss--Bonnet theorem. 
This step allows us to obtain a universal upper bound for the 
norm of the second fundamental form on $\Sigma^\epsilon_0$, which is then used to prove 
the existence of a
limiting hypersurface $\Sigma_0$ that is indeed a component of 
$\partial\Omega$ (Lemma~\ref{convHS}).

Once having an `initial' $\mu$-bubble, one may complete the proof of Theorem~\ref{capRigThm1} by  the foliation argument described above.

Regarding Theorem~\ref{capRigThm2_Intro}, we may consider Riemannian bands in $\S^3\setdiff\{O,O'\}$ bounded by
small geodesic spheres in $(\S^3,\hat g)$ centered at $O$ and $O'$, but
because of the lack of mean curvature information with respect to $g$ along those boundaries,
perturbations of the form \eqref{muepsilon}
are no longer adequate for meeting the barrier condition. To address this issue, we construct
new functions $\mu_\alpha$ by composing the function $\hat H$ with dilations of
 $\S^3\setdiff\{O,O'\}$ in the `longitude' direction, and then $\mu_\alpha$
will satisfy the desired barrier conditions---see Section~\ref{Sec_gen} for more detail. The rest of the proof is similar to the other cases.

%Thanks

\section*{Acknowledgements}

The authors would like to thank Prof. Weiping Zhang for his interest in this work and Dr. Jintian Zhu
for helpful conversations.
Y. Shi and P. Liu are grateful for the support
of the National Key R\&D Program of China Grant 2020YFA0712800,
and Y. Hu is grateful for the support of the China Postdoctoral Science Foundation Grant 2021TQ0014.

%mu-bubbles

\section{Elements of Gromov's $\mu$-bubble technique}\label{Sec_muB}

In this section we recall some elements of Gromov's $\mu$-bubble technique. Our discussion follows Section 5 of \cite{Gromov21Four}, Section 2 of \cite{Zhu21} and Section 3 of \cite{ZZ20}.

\subsection{$\mu$-bubbles in a Riemannian band}\label{Sec_mubb_defs}\

Let $(M^n,g)$ be a compact Riemannian manifold whose boundary $\partial M$ is expressed as a disjoint union $\partial M = \partial_- \sqcup\partial_+$ where both  $\partial_-$ and $\partial_+$ are 
closed hypersurfaces. Such a quadruple $(M,g; \partial_-,\partial_+)$ is called a \emph{Riemannian band}. 
Given a Riemannian band, let  $\Omega_0\subset M$ be a fixed smooth Caccioppoli set%
\footnote{Also known as `sets of locally finite perimeter'; see \cite{Giusti84} for details.}
 that contains a neighborhood of
$\partial_-$ and is disjoint from a neighborhood of $\partial_+$; we call such an $\Omega_0$  
a \emph{reference set}. 
Let $\Ccal_{\Omega_0}$ denote the collection of Caccioppoli sets  
$\Omega \subset M$ satisfying $\Omega\Delta\Omega_0\Subset \mathring M$ (``$\Subset$" reads ``is compactly contained in");
here $\Omega\Delta\Omega_0$ denotes the symmetric difference between $\Omega$ 
and $\Omega_0$, and 
$\mathring M$ stands for the interior of $M$.

Let $\mu$ be either a smooth function on $M$, or a smooth function defined
on $\mathring M$ satisfying $\mu\rightarrow \pm \infty$ on $\partial_{\mp}$.
For $\Omega\in \Ccal_{\Omega_0}$ consider the \emph{brane action}
\begin{equation}\label{braneActDef}
	\Acal^\mu_{\Omega_0}(\Omega) := \Hcal^{n-1}(\partial\Omega) - 
				\Hcal^{n-1}(\partial\Omega_0)
				-\int_M (\chi_{\Omega} - \chi_{\Omega_0})\mu \; d\Hcal^n
\end{equation}
where $\Hcal^k$ is the $k$-dimensional Hausdorff measure induced by $g$ and $\chi_{\Omega}$
denotes the characteristic function associated to $\Omega$. 
A minimizer $\Omega$ of \eqref{braneActDef} is called a \emph{$\mu$-bubble}.

\begin{remark}\label{braneAdditivity}
	(1) For $\Omega_1,\Omega_2\in \Ccal_{\Omega_0}$, we have
$\Acal^\mu_{\Omega_0}(\Omega_2)-  \Acal^{\mu}_{\Omega_1}(\Omega_2)= 
\Acal^{\mu}_{\Omega_0}(\Omega_1)$; thus, in a sense, minimizers are independent of
	  the choice of a reference set.  (2) The brane action \eqref{braneActDef} may be 
	  defined on manifolds that are not necessarily Riemannian bands; in those cases,
	  one may replace $\Hcal^{n-1}(\partial\Omega)$ 
	  by $\Hcal^{n-1}(\partial(\Omega\cap K))$ and similarly for 
	  $\Hcal^{n-1}(\partial\Omega_0)$,
	  where $K$ is a compact set such that
	  $\Omega\Delta\Omega_0\subset K$.
\end{remark}

\subsection{Existence and regularity}\

\begin{definition}\label{barrierDefn}
	Given a Riemannian band $(M,g;\partial_-,\partial_+)$, a function $\mu$ is said to satisfy the  \emph{barrier condition}
	if either $\mu\in C^\infty(\mathring M)$ with $\mu\rightarrow \pm\infty$
	on $\partial_{\mp}$, 
	or $\mu\in C^\infty(M)$ with
	\begin{equation}\label{barrierDef}
		\mu|_{\partial_-} > H_{\partial_-},\qquad \mu|_{\partial_+} < H_{\partial_+}
	\end{equation}
	where $H_{\partial_-}$ is the mean curvature of $\partial_-$ with respect to the inward normal and $H_{\partial_+}$ is the mean curvature of $\partial_+$ with respect to the outward normal.
\end{definition}

\begin{lemma}\label{mubbExReg}{\rm (Cf. \cite[Proposition 2.1]{Zhu21})}
	Let $(M^n,g; \partial_-,\partial_+)$ be a Riemannian band with  $n\le 7$,
	and let $\Omega_0$ be a reference set.
	If $\mu$
	satisfies the barrier condition, 
	then there exists an $\Omega\in \Ccal_{\Omega_0}$ with smooth boundary such that
	\[
		\Acal^\mu_{\Omega_0}(\Omega)
		 = \inf_{\Omega'\in \Ccal_{\Omega_0}} \Acal_{\Omega_0}^\mu(\Omega').
	\]
\end{lemma}

\begin{remark}
	In Lemma~\ref{mubbExReg}
	the smooth hypersurface $\Sigma:=\partial\Omega\setdiff\partial_-$
	is homologous to $\partial_+$. 
\end{remark}

\subsection{Variational properties}\

Let $\Omega$ be a smooth
$\mu$-bubble in a Riemannian band $(M^n,g;\partial_-,\partial_+)$, and let
 $\Sigma = \partial\Omega\setdiff\partial_-$. One may derive 
variation formulae
for $\Acal^\mu$ at $\Omega$---see Equation (2.3) in \cite{Zhu21} and the unnumbered equation above it. 
Specifically,
the first variation implies that the mean curvature of $\Sigma$ (with its outward normal $\nu$) is equal
to $\mu|_{\Sigma}$;
the second variation implies that 
the Jacobi operator 
\begin{equation}\label{Jdefn}
    J_\Sigma :=  -\Delta_\Sigma + \frac{1}{2} (R_\Sigma - R_g - \mu^2 - |\II|^2) - \nu(\mu)
\end{equation}
is non-negative,
where 
$\Delta_\Sigma$ and  $R_\Sigma$ are respectively the $g$-induced Laplacian and scalar curvature of $\Sigma$; $R_g$ is the scalar curvature of $(M,g)$; and
$\II$ is the second fundamental form of $\Sigma$. 

\begin{definition}\label{MhsDefn}
Let $\mu$ be a smooth function on a Riemannian manifold $(M^n,g)$.
A smooth two-sided hypersurface $\Scal\subset M$ with unit normal $\nu$
	is said to be a \emph{$\mu$-hypersurface} if its mean curvature taken with respect to $\nu$
	is equal to $\mu|_{\Scal}$.
\end{definition}

Clearly, \eqref{Jdefn} also makes sense when $\Sigma$ is replaced by a  $\mu$-hypersurface; this motivates the following notion of stability.
\begin{definition}\label{stabilityDefn}
	A $\mu$-hypersurface  $\Scal\subset M$ with unit normal $\nu$ is said to be
	\emph{stable} if $J_\Scal$ is non-negative on $C^\infty_0(\Scal)$.
\end{definition}

\begin{remark}
	If $\mu$ satisfies the barrier condition, then for any $\mu$-bubble $\Omega$ each connected component of
	$\partial\Omega\setdiff\partial_-$ with its 
	outward unit normal is a stable $\mu$-hypersurface.
\end{remark}

Let $\Scal$ be a $\mu$-hypersurface. Following \cite[Section 5.1]{Gromov21Four} we consider the operator 
\begin{equation}\label{Ldefn}
    L_\Scal:= -\Delta_\Scal +\frac{1}{2} (R_\Scal
    -R_+^\mu)
\end{equation} 
where
\begin{equation}\label{Rplus}
    R_+^\mu := 
    R_g + \frac{n}{n-1}\mu^2 - 2|\ed\mu|_g.
\end{equation}
In fact, $L_\Scal$ is obtained from applying the obvious inequalities
\begin{equation}\label{obviousIneq}
    -\partial_{\nu} \mu \le |\ed \mu|_g,
    \qquad|\II|^2 \ge\frac{1}{n-1}\mu^2
\end{equation}
to $J_\Scal$. One can easily verify that the following holds when $\Scal$ is stable:
\begin{equation}\label{JLcompare}
    L_\Scal\ge J_\Scal\ge 0.
\end{equation}

\begin{example}
	Consider $\S^2\times [t_1,t_2]$ $(0<t_1<t_2<\pi)$ equipped with the metric
	$g = (\sin^2t)g_{\S^2}  + \ed t^2$ where $g_{\S^2}$ is the standard metric on $\S^2$.
	This represents an annular region in the standard $\S^3$.
	Take $\mu(t) = 2\cot t$. It is easy to see that each $t$-level set $S_t$, with the unit normal $\nu = \partial_t$, is a $\mu$-hypersurface. Moreover, on $S_t$
	we have
	\[
		R_{g} = 6, \quad R_{S_t} = \frac{2}{\sin^2t}, \quad 
		|\II|^2 = 2\cot^2t,  \quad \nu(\mu)= \mu'(t) = -\frac{2}{\sin^2 t}. 
	\]
	In this case, both $J_{S_t}$ and $L_{S_t}$ reduce to $-\Delta_{S_t}$.
\end{example}

The following lemma is part of Theorem 3.6 in \cite{ZZ20}.
\begin{lemma}\label{sffLemma}
	Let $(M^n,g)$ be a closed Riemannian manifold with $2\le n\le 6$, and let $\mu\in C^\infty(M)$. 
	Let $\Scal$ be an immersed stable $\mu$-hypersurface
	 contained in an open subset $V\subset M$ and satisfying $\partial\Scal\cap V = \emptyset$.
	 If $\area(\Scal)\le C$ for some constant $C$, 
	 then there exists a constant $C_1 = C_1(M,n,\|\mu\|_{C^3(M)}, C)$ such that 
	\begin{equation}\label{curvatureBD}
		|\II|^2(x) \le \frac{C_1}{\dist_g^2(x,\partial V)} \quad \mbox{for all } x\in \Scal.
	\end{equation}
\end{lemma}

\subsection{Comparison with a warped-product metric}\

Given a  Riemannian manifold $(N^{n-1}, g_N)$, an interval $I$ (with coordinate $t$) and a function $\varphi: I\rightarrow \R_+$, consider  the warped
product metric defined on $\hat N:=N\times I$
\begin{equation}\label{g_warp_defn}
    \hat g := \varphi(t)^2 g_N + \ed t^2.
\end{equation}
A standard calculation shows that the mean curvature on each slice $N\times\{t\}$ with respect to the $\partial_t$-direction is
\begin{equation}\label{g_warp_mean}
    \hat H(t) = (n-1)\frac{\varphi'(t)}{\varphi(t)};
\end{equation}
moreover, one may verify that the scalar curvature $R_{\hat g}$ of $\hat g$ satisfies 
\begin{equation}\label{Rhat}
    0  = - R_{\hat g} + \frac{1}{\varphi^2} R_N
            - \frac{n}{n-1}\hat H^2 - 2\frac{d\hat H}{d t},
\end{equation}
where $R_N$ is the scalar curvature of $(N,g_N)$.

Now suppose that $f: M\rightarrow \hat N$ is a smooth map from a Riemannian band $(M,g;\partial_-,\partial_+)$ to $\hat N$.
By pulling back all functions in \eqref{Rhat} via $f$ and adding the resulting equation with \eqref{Rplus}, we obtain
\begin{equation}\label{RplusEst}
    R_+^\mu  =\frac{1}{\varphi^2}R_{N} +  (R_g - R_{\hat g})
            + \frac{n}{n-1}(\mu^2 - \hat H^2) -2 \left(\partial_t\hat H + |\ed\mu|_g\right),
\end{equation}
where pull-back symbols are omitted for clarity.
The expression \eqref{RplusEst} will be useful in our analysis of $\mu$-bubbles.

%mu-bubble to rigidity

\section{Rigidity of 3D spherical caps}\label{rigSection}

A spherical cap of radius $T\in(0,\pi)$ in the standard $\S^3$ may be represented by the closed ball 
$\B_T :=\{\xb\in \R^3: |\xb|\le T\}$ equipped with the metric
\begin{equation}\label{capMetricStd}
	\hat g =  \varphi(t)^2 g_{\S^2} + \ed t^2 \quad\mbox{with } \varphi(t) = \sin t,
\end{equation}
where $t\in [0,T]$ serves as the radial coordinate on $\B_T$ and $g_{\S^2}$ is the standard metric on $\S^2$.
For $t\in(0,T]$, let $S_t := \partial \B_t$. For $0<t_1<t_2\le T$, let  
$\B_{[t_1,t_2]}:=\B_{t_2} \setdiff \mathring\B_{t_1}$; similarly, let $\B_{(t_1,t_2]}:=\B_{t_2}\setdiff \B_{t_1}$. Given a domain $\Omega\subset \B_T$ with smooth boundary $\Sigma$,
the outward normal along $\Sigma$ with respect to the metric $\hat g$ will be denoted by $\hat \nu$.

The objective of this section and the next is to prove the following rigidity theorem.

\begin{theorem}\label{capRigThm}
	Let $(\B_T, \hat g)$ be the 3-dimensional spherical cap of radius $T\in (0,\pi)$. Suppose that $g$ is another Riemannian
	metric on $\B_T$ satisfying 
	\begin{equation}\label{capMetricCond}
		g\ge \hat g, \quad R_g\ge R_{\hat g} \mbox{ on } \B_T,\quad  \mbox{and } 
		H_g\ge H_{\hat g} = 2\cot T \mbox{ on } \partial\B_T.
	\end{equation}
	Then $g = \hat g$.
\end{theorem}

Our proof begins by establishing a key ingredient: 
certain stable $\mu$-hypersurfaces are necessarily $t$-level sets in $\B_T$
(Proposition~\ref{levelsetProp}), the justification of which hinges on
an integral inequality (see \eqref{Rplus_intEst}) involving an application of 
the Gauss--Bonnet formula.
This result is followed by a classical foliation lemma (Lemma~\ref{foliationLemma}).
Under a suitable `minimality' assumption (Assumption~\ref{relativeCCassum}), 
each leaf in that foliation
turns out to be stable, which implies local rigidity of the metric (Lemma~\ref{localRigLemma}).
Section~\ref{rigSection} culminates at Proposition~\ref{rigidityProp}, which justifies 
Theorem~\ref{capRigThm} assuming the existence of an `initial' minimizer
(Assumption~\ref{relativeCCassum}); this assumption will be verified in 
Section~\ref{Sec_initMB} via a perturbation argument (see Proposition~\ref{initProp}).

\subsection{Stable $\mu$-hypersurfaces and $t$-level sets}\

The metric \eqref{capMetricStd} is of the form \eqref{g_warp_defn}; thus, 
\eqref{g_warp_mean} applies to give
\begin{equation}\label{HhatDef}
	\hat H(t) = 2\cot t.
\end{equation}
It will be useful to define, for $\mu = \mu(t)$, the function
(cf. the last two terms in \eqref{RplusEst})
\begin{equation}\label{Zdefn}
	\begin{alignedat}{1}
	Z_{\mu}(t)& := \frac{3}{2}(\mu(t)^2 - \hat H(t)^2) -2 (\hat H'(t) - \mu'(t))\\
			&  = \frac{3}{2}\mu(t)^2 + 2\mu'(t) - 6\cot^2 t + \frac{4}{\sin^2t}.
	\end{alignedat}		
\end{equation}
Notice, in particular, that $Z_{\hat H}(t)\equiv 0$.
As $t$ is a coordinate on $\B_T$, we may regard $\mu$ and $Z_\mu$ as functions
defined on $\B_T\setdiff\{\zerob\}$.

\begin{lemma}\label{RplusLemma}
	Let $\mu(t)$ be a smooth, decreasing function defined on $(0, T]$, and let 
	$g$ be a Riemannian metric on $\B_T$ satisfying \eqref{capMetricCond}. 
	At a point $q\in \B_T$, if $Z_\mu \ge 0$,
	 then $R_+^{\mu} \ge 2/\varphi^2 >0$.
\end{lemma}
\begin{proof}
On the RHS of \eqref{RplusEst}, the second term is non-negative by assumption.
 Moreover, $g\ge \hat g$ implies
\begin{equation}
    |\ed\mu|_g \le |\ed\mu|_{\hat g} =  |\partial_t\mu| = -\mu'(t).
\end{equation}
Substituting this in the last term of \eqref{RplusEst} and noticing that $R_{\S^2} = 2$, we obtain 
the desired inequality.
\end{proof}

Now let $\Sigma_0$ be a hypersurface in $\B_{(0, T]}$, and let $\Phi$ denote the projection 
map from $\Sigma_0$ to $\S^2$, namely,
\begin{equation}\label{projPhiDef}
	\Phi: \Sigma_0\hookrightarrow \B_{(0, T]} \cong (0, T]\times\S^2\xrightarrow{} \S^2.
\end{equation}

\begin{lemma}\label{areaCompLemma}
	Let $d\sigma_{\hat g}$ be the area form on $\Sigma_0$ induced by $\hat g$.
	 We have 
	\begin{equation}
		d\sigma_{\hat g} \ge \varphi^2 |\Phi^* d\sigma_{\S^2}|
	\end{equation}
	where the absolute-value sign is put to eliminate the ambiguity of orientation.
\end{lemma}
\begin{proof}
Let $(\theta_\alpha)$ $(\alpha = 1,2)$ be local coordinates on $\S^2$, and write $g_{\S^2} = h_{\alpha\beta}\ed\theta_\alpha \ed\theta_\beta$. We get
    \begin{equation}
        \Phi^*(g_{\S^2}) = h_{\alpha\beta}\ed\theta_\alpha\ed\theta_\beta \le \frac{1}{\varphi^2} \left(\ed t^2 + \varphi^2 h_{\alpha\beta}\ed \theta_\alpha \ed\theta_\beta\right) = \frac{1}{\varphi^2} \hat g_{\Sigma_0},
    \end{equation}
    where the functions and forms are restricted to $\Sigma_0$. The conclusion follows.
\end{proof}

\begin{prop}\label{levelsetProp}
Let $\mu(t)$ be a smooth, decreasing function defined on $(0,T]$.
   Suppose that $\Sigma_0\hookrightarrow(\B_T\setdiff\{\zerob\}, g)$ is 
   a stable, closed $\mu$-hypersurface with
   unit normal $\nu$,
   where $g$ satisfies $g\ge \hat g$ and $R_g\ge R_{\hat g}$.
   Moreover, suppose that $Z_\mu\ge 0$ on $\Sigma_0$ and that
   the projection $\Phi$ from $\Sigma_0$ to $\S^2$ has nonzero degree.
   Then
\begin{enumerate}[\rm(a)] 
\item{$\Sigma_0 = S_\tau$ for some $\tau\in (0, T]$; 
		\label{levelset_ls}}
\item{  $J_{\Sigma_0} = L_{\Sigma_0} = -\Delta_{\Sigma_0}$ (see \eqref{Jdefn}, \eqref{Ldefn});
		\label{levelset_jacobi}}
\item{$\Sigma_0\subset (\B_T, g)$ is umbilic with constant mean curvature $\mu(\tau)$;
		\label{levelset_umb}}
\item{$g(p) = \hat g(p)$ at all points $p\in \Sigma_0$; in particular, 
		$g_{\Sigma_0} = \hat g_{\Sigma_0} = (\sin^2\tau)  g_{\S^2}$;
		\label{levelset_metric}}
\item{on $\Sigma_0$, $\nu = \partial_t$;
		\label{levelset_normal}}
 \item{on $\Sigma_0$, $R_+^\mu = 2/\varphi^2$ and $Z_\mu = 0$.
 		\label{levelset_Z}}
\end{enumerate}
\end{prop}

We prepare our proof of this proposition with the following two lemmas.
\begin{lemma}\label{sphereLemma}
    Under the assumption of Proposition~\ref{levelsetProp}, 
    $\Sigma_0$ is homeomorphic to $\S^2$.
\end{lemma}

\begin{proof}
    By stability, the operator $L_{\Sigma_0}$ defined by \eqref{Ldefn} is non-negative. Let $u\in C^\infty(\Sigma_0)$ be a principal eigenfunction of $L_{\Sigma_0}$, and let $\lambda_1\ge 0$ be the corresponding eigenvalue. By the maximum principle, we can always choose $u$ to be strictly positive. Thus,
    \begin{equation}\label{Lfirstev}
        -u^{-1}\Delta_{\Sigma_0} u + \frac{1}{2}(R_{\Sigma_0} - R_+^\mu) = \lambda_1\ge 0.
    \end{equation}
    Expanding 
    \begin{equation}
    \div(u^{-1}\nabla_{\Sigma_0} u) = - u^{-2}|\nabla_{\Sigma_0} u|^2 + u^{-1}\Delta_{\Sigma_0} u,
    \end{equation} 
    applying it in the previous equation and integrating over $\Sigma_0$, we obtain
    \begin{equation}\label{RminusRplusIntEst}
        \frac{1}{2}\int_{\Sigma_0} (R_{\Sigma_0} - R_+^\mu) d\sigma_g
         = \int_{\Sigma_0} (\lambda_1 + u^{-2}|\nabla_{\Sigma_0}u|^2) d\sigma_g\ge 0.
    \end{equation}
    From \eqref{RminusRplusIntEst}, the Gauss--Bonnet formula, and Lemma~\ref{RplusLemma}, we deduce
    \begin{equation}\label{GaussBonnet}
        4\pi\chi(\Sigma_0) = \int_{\Sigma_0} R_{\Sigma_0} d\sigma_g \ge \int_{\Sigma_0} R_+^\mu d\sigma_g > 0;
    \end{equation}
     since $\Sigma_0$ is a connected oriented surface, it is homeomorphic to $\S^2$.
\end{proof}

\begin{remark}\label{ZepsilonRelaxRmk}
	Lemma~\ref{sphereLemma} remains true if  we assume $R^\mu_+>0$ instead of
	$Z_\mu\ge 0$ on $\Sigma_0$.
\end{remark}

\begin{lemma}\label{levelSetLemma}
	Under the assumption of Proposition~\ref{levelsetProp}, if 
	$J_{\Sigma_0} = L_{\Sigma_0}=-\Delta_{\Sigma_0}$, then
	\begin{enumerate}[\rm (i)]
		\item $\Sigma_0\subset (\B_T,g)$ is umbilic;
		\item $\Sigma_0 = S_\tau$ for some $\tau\in (0,T]$;
		\item  $\mu|_{\Sigma_0} =  \mu(\tau)$.
	\end{enumerate}
\end{lemma}
\begin{proof}
	By assumption, \eqref{obviousIneq} must be equalities. 
	In particular, the traceless part of $\II_{\Sigma_0}$ must vanish, and thus $\Sigma_0\subset (\B_T,g)$ is umbilic, justifying (i).
     	Moreover, $-\nu(\mu)= |\ed \mu|_g$, and so $\nu$ must be 
	parallel to $\nabla_g \mu$. Thus,
     for any tangent vector $X\in T\Sigma_0$, we have that $\ed\mu(X) = g(\nabla_g\mu, X)$  
     is proportional to $g(\nu,X) = 0$; this implies that
      $\mu$ is constant along $\Sigma_0$. Combining with the fact that $\Sigma_0\cong \S^2$ (Lemma~\ref{sphereLemma}), we conclude that $\Sigma_0$ is a level set $S_\tau$, justifying (ii), and
      (iii) immediately follows.
\end{proof}

\noindent\emph{Proof of Proposition~\ref{levelsetProp}.}
    The assumption $g\ge \hat g$ implies the relation between area forms on $\Sigma_0$:
    \begin{equation}\label{sigmaVSsigma}
        d\sigma_g \ge d\sigma_{\hat g}.
    \end{equation}    
  We deduce
    \begin{equation}\label{Rplus_intEst}
    \begin{alignedat}{1}
        \int_{\Sigma_0}R_+^\mu d\sigma_g
       &\ge \int_{\Sigma_0} \frac{2}{\varphi^2} d\sigma_{\hat g}
        \ge 2 \int_{\Sigma_0}  |\Phi^*d\sigma_{\S^2}|\\
        &\ge 2\left|\int_{\Sigma_0} \Phi^*d\sigma_{\S^2}\right| 
        =2 k \int_{\S^2}d\sigma_{\S^2} = 8k\pi,
    \end{alignedat}
    \end{equation}
    where $k := |\deg(\Phi) |\ge 1$ by assumption.  
    In \eqref{Rplus_intEst}, the first inequality is due to \eqref{sigmaVSsigma} and 
     Lemma~\ref{RplusLemma}; the second inequality follows from Lemma~\ref{areaCompLemma}; the remaining (in)equalities are obvious.
     
     On combining \eqref{GaussBonnet} with \eqref{Rplus_intEst}, we obtain
     \begin{equation}\label{GaussBonnetCbd}
         8\pi = \int_{\Sigma_0}R_{\Sigma_0}d\sigma_g \ge \int_{\Sigma_0} R_+^\mu d\sigma_g
         \ge 8k\pi, \qquad (k\ge 1).
     \end{equation}
     This enforces the two inequalities in \eqref{GaussBonnetCbd} to become equalities.
     Saturation of the first inequality, which we deduced from \eqref{RminusRplusIntEst}, implies that
     $\lambda_1 = 0$ and that $u$ is a constant; hence, by \eqref{Lfirstev}, $R_{\Sigma_0} = R_+^\mu$;
      then, by \eqref{Ldefn}, $L_{\Sigma_0} = -\Delta_{\Sigma_0}$.  
    With this established, the relation \eqref{JLcompare} would enforce that $J_{\Sigma_0} = L_{\Sigma_0} = -\Delta_{\Sigma_0}$, justifying (\ref{levelset_jacobi}). By Lemma~\ref{levelSetLemma}, (\ref{levelset_ls}) and (\ref{levelset_umb})
    follow.
  
  Next consider saturation of the second inequality in \eqref{GaussBonnetCbd}, or rather \eqref{Rplus_intEst}. Because we have already deduced that $\Sigma_0$ is a $t$-level set, the
  second and third inequalities in \eqref{Rplus_intEst} automatically become equalities. Saturation of the first inequality in \eqref{Rplus_intEst}, on the other hand, has two implications:
  \[
  	d\sigma_g = d\sigma_{\hat g}\quad \mbox{and} \quad
	R_+^\mu = \frac{2}{\varphi(\tau)^2}.
  \]
 The former, along with $g\ge \hat g$, implies that
 \begin{equation}\label{eqIndMet}
 g_{\Sigma_0} = \hat g_{\Sigma_0} = \varphi(\tau)^2g_{\S^2};
 \end{equation}
  the latter, 
 along with the proof of Lemma~\ref{RplusLemma}, implies that $Z_\mu(\tau) = 0$ and 
 $
 	|\ed\mu|_g = |\ed\mu|_{\hat g},
$	
which is just $-\nu(\mu) = |\partial_t\mu|$ (see the proof of Lemma~\ref{levelSetLemma}). 
Hence, $\nu = \partial_t + X$
for some vector field $X$ on $\Sigma_0 = S_\tau$. Note that
\begin{equation}
	1 = |\nu|_g \ge |\nu|_{\hat g} = |\partial_t|_{\hat g} + |X|_{\hat g} = 1+ |X|_{\hat g};
\end{equation}
we have $X = 0$ and $\nu= \partial_t$. Combining this with \eqref{eqIndMet}, we get $g(p) = \hat g(p)$ for all $p\in \Sigma_0$. This justifies (\ref{levelset_metric}), (\ref{levelset_normal}) and (\ref{levelset_Z}), completing the proof.
\qed

\subsection{Foliation, minimality and rigidity}\

The following `foliation' lemma is standard (cf. \cite{Ye91}, \cite{ACG08}, \cite{Nunes13} and \cite{Zhu21}).

\begin{lemma}\label{foliationLemma}
	Suppose that $\Sigma_0\subset (\B_T,g)$ is a $\mu$-hypersurface (with unit normal $\nu$)
	on which the stability operator $J$ (see \eqref{Jdefn}) reduces to $-\Delta_{\Sigma_0}$.	
	Then there exists an interval\,%
	 \footnote{If $0<\tau<T$, $I$ can be taken to be an open interval containing $0$; if $\tau = T$, $I$ is of the form $(a ,0]$; and if $\tau = \delta$, $I$ is of the form $[0, b)$.}
	 $I$ and a map $\phi: \Sigma_0\times I\rightarrow \B_T$ such that
	\begin{enumerate}[\rm (1)]
		\item $\phi$ is a diffeomorphism onto a neighborhood of $\Sigma_0\subset \B_T$;
		\item the family $\Sigma_s = \phi(\Sigma_0, s)$ is a normal variation of $\Sigma_0$ with
			$\partial_s\phi =\nu$ along $\Sigma_0$;  and
	 	\item on each $\Sigma_s$, the difference $H_{\Sigma_{s}} - \mu$ is a constant $k_s$.
	\end{enumerate}
	\end{lemma}
\begin{proof} The proof is the same as that of Lemma~3.4 in \cite{Zhu21}, except for 
the extra step: once having obtained a foliation, we re-express it as a normal variation by using a vector field
normal to all its leaves (cf. \cite[p.6, 2nd paragraph]{ACG08}). 
\end{proof}
	
Before proceeding further, let us state a recurring assumption.

\begin{assumption} \label{relativeCCassum}
	Let $g$ be a metric on $\B_T$ satisfying \eqref{capMetricCond}, and let
	$\Omega\subset (\B_T,g)$ be a Caccioppoli set such that $\partial\Omega\setdiff\{\zerob\}$ 
	is smooth and embedded. Define the class $\Ccal_\Omega$ of 
	Caccioppoli sets by
	\begin{equation}\label{varCls}
		\Ccal_{\Omega} :=\left\{\Omega'\subset \B_T \mbox{ Caccioppoli set}:
						\Omega'\Delta \Omega \Subset \B_T\setdiff\{\zerob\}\right\}.
	\end{equation} 
	Suppose that $\Omega$ is a minimizer in the sense that
	 for any $\Omega'\in\Ccal_\Omega$, we have
 	$\Acal^{\hat H}_{\Omega}(\Omega')\ge 0$;
	and assume that there is a connected component 
	$\Sigma_0\subset \partial\Omega$
	that is a stable $\hat H$-hypersurface,%
	\footnote{We allow $\Sigma_0$ to overlap with $\partial\B_T$.}
	disjoint from $\zerob\in \B_T$ and with nonzero-degree
	projection onto $\S^2$. 
	
\end{assumption}

\begin{lemma} {\rm (Cf. \cite[Section~5.7]{Gromov21Four})}\label{localRigLemma}
	If Assumption~\ref{relativeCCassum} holds, then 
	\begin{enumerate}[\rm (1)]  
	\item there exists a constant $\tau\in (0, T]$ such that
	$\Sigma_0 = S_\tau$ with outward normal $\partial_t$; and
	\item there exists an open neighborhood
	$\Ucal$ of $\Sigma_0 = S_\tau$, disjoint from $\partial\Omega\setdiff\Sigma_0$,
		 on which $g = \hat g$.
	\end{enumerate}
\end{lemma}

\begin{proof}
	 Since $\Sigma_0$ is assumed to be a stable, closed 
	 $\hat H$-hypersurface, and
	 since $Z_{\hat H}\equiv 0$ (see \eqref{Zdefn}), 
	 Proposition~\ref{levelsetProp} applies and yields (1).
	  	
	To prove (2), first note that Proposition~\ref{levelsetProp} and Lemma~\ref{foliationLemma}
	together imply that a neighborhood $\Ucal$ of $\Sigma_0$ is foliated by 
	a normal variation $\{\Sigma_s\}$ $(s\in I)$ of $\Sigma_0$; moreover, on each leaf $\Sigma_s$
	the difference $H_{\Sigma_s} -\hat H$ is 
	a constant $k_s$. Since $\zerob\notin\Sigma_0$ and 
	$\dist_g(\Sigma_0, \partial\Omega\setdiff\Sigma_0)>0$, $\Ucal$ can be chosen 
	to be disjoint from both $\partial\Omega\setdiff\Sigma_0$ and $\zerob$.
	
	For $s_1, s_2\in I$ with $s_1<s_2$ define $\Sigma_{[s_1,s_2]}\subset \B_T$ 
	to be the (compact) subset 
	with boundary $\Sigma_{s_1}\cup \Sigma_{s_2}$; then consider $\Omega_s$ defined by
	 \begin{equation}
	 	\Omega_s:=\left\{
		\begin{alignedat}{2}
		& \Omega\cup \Sigma_{[0,s]}&&, \quad\mbox{if } s\ge 0,\\
		 &\Omega\setdiff\Sigma_{[-s,0]} &&,\quad\mbox{if } s<0.
		\end{alignedat}	
		\right.
	\end{equation}
	Clearly, these
	$\Omega_s$ belong to the class $\Ccal_\Omega$.
	Let us denote
	$\Acal^{\hat H}_\Omega(\Omega_s)$ by $\Acal(s)$ for brevity, and write
	$u_s = \<\partial_s\phi,\nu_s\>>0$ where $\nu_s$ is the (suitably oriented) unit normal 
	along $\Sigma_s$. By Lemma~\ref{foliationLemma} and the first variation formula,
	\begin{equation}\label{firstVarOnFoliation}
		\Acal'(s) = \int_{\Sigma_s} k_s u_s.
	\end{equation}
	Since $\Acal(0)$ attains the minimum, it is necessary that 
	\begin{enumerate}[(i)]
		\item either $\Acal(s)\equiv 0$ for all $s\ge 0$, 
					or $\Acal'(s)>0$ (equivalently, $k_s>0$) for some $s>0$; 
		\item either $\Acal(s)\equiv 0$ for all $s\le 0$, 
					or $\Acal'(s)<0$  (equivalently, $k_s<0$) for some $s<0$.
	\end{enumerate}	
	To complete the proof, it suffices to show that $\Acal(s)\equiv 0$ for all $s\in I$. 
	If this does not hold, first suppose that $k_{s}>0$ for some $s>0$. Then on the Riemannian band
	$\Sigma_{[0,s]}$ with $\partial_- = \Sigma_0$ and $\partial_+=\Sigma_s$ define the function 
	\begin{equation}\label{mutilde}
		\tilde\mu(t) = \hat H(t) + \frac{\epsilon}{\sin^3t},
	\end{equation}
	which is smooth and decreasing in $t$. By choosing sufficiently small $\epsilon$, we can arrange that $\tilde\mu>H_{\Sigma_0}$ on $\Sigma_0$ and that 
	$\tilde\mu< H_{\Sigma_s}$ on $\Sigma_s$.
	Thus, by Lemma~\ref{mubbExReg}, there exists a $\tilde\mu$-bubble $\tilde\Omega$ in $\Sigma_{[0,s]}$; in particular, $\tilde\Sigma = \partial\tilde\Omega\setdiff\Sigma_0$ has a component $\tilde\Sigma_0$ whose projection to $\S^2$ has nonzero degree. However, by a direct calculation
	using \eqref{Zdefn}, we get
	\begin{equation}
		Z_{\tilde\mu}(t) = \frac{3\epsilon^2}{2\sin^6 t} > 0,
	\end{equation}
	contradicting Proposition~\ref{levelsetProp}(\ref{levelset_Z}).
	
	The case when $k_s<0$ for some $s<0$ may be similarly and independently 
	ruled out; it suffices to consider
	$\Sigma_{[s,0]}$ with
	$\partial_-=\Sigma_s$ and $\partial_+ = \Sigma_0$ and
	 the following analogue of \eqref{mutilde}: 
	$\tilde\mu(t) = \hat H(t) - \epsilon \sin ^{-3}t$.
	
	Finally, since we have proved that all $\Omega_s$ are $\Acal^{\hat H}$-minimizing
	in the class $\Ccal_{\Omega}$, each $\Sigma_s$ must be a $t$-level set. By 
	Proposition~\ref{levelsetProp}(\ref{levelset_metric}), $g = \hat g$ on $\Ucal$, and this completes the proof.
\end{proof}

\begin{prop}\label{rigidityProp}
	If Assumption~\ref{relativeCCassum} holds, 
	then $g = \hat g$ on $\B_T$.
\end{prop}

\begin{proof}
	By Lemma~\ref{localRigLemma}, $\Sigma_0 = S_\tau$ for some $\tau\in (0,T]$, and its
	outward normal is $\partial_t$.
	Without loss of generality, we assume $\tau\in (0,T)$. Let 
	$I = (t_1,t_2)$ be the maximum open interval containing $\tau$ 
	such that $\B_{(t_1,t_2)}$ is disjoint from $\partial\Omega\setdiff\Sigma_0$
	and that $g = \hat g$ on  $\B_{(t_1,t_2)}$. For $t\in I$, let 
	$\Omega_t$ denote $\Omega\setdiff \B_{(t,\tau]}$ if $t< \tau$
	and $\Omega\cup \B_{[\tau,t]}$ if $t\ge \tau$. In particular, 
	$\partial\Omega_t = (\partial\Omega\setdiff\Sigma_0)\cup S_t$.

	It suffices to show that $t_1 = 0$ and $t_2 = T$, and we argue by contradiction.
	First suppose that $t_1>0$. 
	Then
	$\Omega_{t_1}$ is in the class $\Ccal_{\Omega}$, and it satisfies 
	$\Acal^{\hat H}_{\Omega}(\Omega_{t_1}) = 0$.
	If $S_{t_1}$ were disjoint from $\partial\Omega\setdiff\Sigma_0$, then, 
	by Lemma~\ref{localRigLemma}, the interval $I$ can be extended further, 
	violating its maximality. On the other hand, if $S_{t_1}$
	were to intersect a connected component 
	$\Sigma'\subset \partial\Omega\setdiff\Sigma_0$, 
	then by smoothness and embeddedness 
	$\partial\Omega_{t_1}\setdiff\{\zerob\}$ (cf. \cite[Theorem 2.2]{ZZ20}), 
	$\Sigma'$ must be equal to $\Sigma_{t_1}$ but 
	with the opposite outward
	normal, violating Proposition~\ref{levelsetProp}(\ref{levelset_normal}).
	Thus, we conclude that $t_1 = 0$.
	The proof of $t_2 = T$ is similar.
\end{proof}

With Proposition~\ref{rigidityProp}, it becomes clear that 
Theorem~\ref{capRigThm} would follow if one can verify
Assumption~\ref{relativeCCassum}.
To illustrate this point, we now discuss a special case
of Theorem~\ref{capRigThm} which admits a more direct proof.
(The general situation is more subtle and will be addressed in the next section.)

\begin{cor}\label{bdryMetricCor}
	Let $(\B_T, \hat g)$ be the 3-dimensional spherical cap of radius $T\in (0,\pi/2]$. Suppose that $g$ is another Riemannian
	metric on $\B_T$ satisfying  
	$g \ge \hat g$ and
	$R_g\ge R_{\hat g}$
	on $\B_T$; in addition, suppose that
	$H_g\ge H_{\hat g} = 2\cot T$
	and
	$g_{\partial\B_T}  = \hat g_{\partial\B_T}$
	on $\partial\B_T$.
	Then $g = \hat g$.
\end{cor}

	\begin{proof} 
		Take $\mu = \hat H$, which is in $L^1(\B_T)$. 
		Since adding a constant to a functional
		does not affect its variational properties, we may consider, instead of \eqref{braneActDef},
		\begin{equation}
		\Bcal^\mu(\Omega) := \Hcal^{2}(\partial\Omega)  + \int_{\B_{T}\setdiff\Omega} \mu d\Hcal^3,
		\end{equation}
		for all smooth Caccioppoli sets $\Omega\subset \B_T$ with 
		$\Omega\Delta\B_T\Subset \B_T\setdiff\{\zerob\}$, and underlying metrics
		will be specified in subscripts. 
		Since $\hat H = \div (\partial_t)$ on $\B_{T}\setdiff\{\zerob\}$,
		we have
		\begin{equation}\label{BTMinimizing}
			\Bcal^\mu_{\hat g}(\Omega) = \Hcal^2_{\hat g}(\partial\Omega) - 
						\int_{\partial\Omega} \<\partial_t,\hat\nu\>_{\hat g} d\Hcal_{\hat g}^2 + \Hcal^2_{\hat g}(S_T) \ge \Hcal^2_{\hat g}(S_T)
					= \Bcal^\mu_{\hat g}(\B_T),
		\end{equation}
		where the first equality is an application of the divergence formula, and the inequality
		is derived from the relation $\<\partial_t, \hat\nu\>_{\hat g}\le 1$.
		Now, since $g\ge \hat g$ and $\mu\ge 0$ on $\B_T$ 
		$(T\le\pi/2)$, we have
		$\Bcal_g^\mu(\Omega)\ge \Bcal_{\hat g}^\mu(\Omega)$; moreover, by
		$g_{\partial\B_T}  = \hat g_{\partial\B_T}$, we have
		$\Bcal_{\hat g}^\mu(\B_{T})  = \Bcal_{g}^\mu(\B_{T})$. Combining these with 
		\eqref{BTMinimizing} gives
		$
			\Bcal_g^\mu(\Omega)\ge 
			\Bcal_{g}^\mu(\B_{T});
		$
		 and using $H_g\ge 2\cot T = \hat H|_{\partial\B_T}$, we deduce that
		 $H_g = 2\cot T$ and that $S_T$ is a stable $\hat H$-hypersurface.
		 Now it is easy to see that the pair $(\B_T, S_T)$ satisfies 
		 Assumption~\ref{relativeCCassum}. The conclusion then follows from
		Proposition~\ref{rigidityProp}.
	\end{proof}

% Initial mu-bubble

\section{Existence of an initial minimizer}\label{Sec_initMB}

Throughout this section, let $g$ be a Riemannian metric on $\B_T$ satisfying \eqref{capMetricCond}.
Our goal is to obtain an `initial' minimizer $\Omega$ and a connected 
component $\Sigma_0\subset\partial\Omega$ 
which satisfy Assumption~\ref{relativeCCassum}. 
To achieve this, we consider perturbations $\mu_\epsilon$
of $\hat H  = 2\cot t$ (see \eqref{muepsilon}).
For each $\epsilon$, we find a Riemannian band $M_\epsilon\subset \B_T$
on which $\mu_\epsilon$ 
satisfies the barrier condition; thus, a $\mu_\epsilon$-bubble 
$\Omega_\epsilon$ exists, and
$\partial\Omega_\epsilon\cap\mathring M_\epsilon$ 
has a component $\Sigma^\epsilon_0$ which projects onto $\S^2$ with 
nonzero degree. One may wonder whether this `degree' property is preserved
in the limit as $\epsilon\rightarrow 0$; this led us to
find that each $\Sigma^\epsilon_0$
must be disjoint from a fixed open neighborhood of $\zerob\in \B_T$, provided $\epsilon$ is small (Proposition~\ref{noCrossing}). Then we verify Assumption~\ref{relativeCCassum} by
analyzing the limits of $\Omega_\epsilon$ and $\Sigma^\epsilon_0$ (Proposition~\ref{initProp}).

\subsection{A choice of $\mu_\epsilon$}\

Let $\epsilon>0$ be a small constant, and define
\begin{equation}\label{tcDef}
    t_c := \min\left\{ \frac{\pi}{4}, \frac{T}{2}\right\}.
\end{equation}
Moreover, we shall fix a
function $\beta\in C^\infty((0,T])$ which is strictly decreasing and satisfies
\begin{equation}\label{betaDefn}
	\beta(t) =  \cot t  \mbox{ on } (0, t_c] \quad \mbox{and }\quad 
	\beta(T) = -1;
\end{equation}
such a $\beta$ clearly exists.
Now consider the function defined on $(0,T]$: 
\begin{equation}\label{muepsilon}
    \mu_\epsilon(t) \equiv \hat H(t) + \epsilon\beta(t) = 2\cot t + \epsilon \beta(t).
\end{equation}
Writing $Z^\epsilon$ for $Z_{\mu_\epsilon}$, we have  (cf. \eqref{Zdefn})
\begin{equation}\label{Zepsilon}
	Z^\epsilon(t) = \frac{3}{2}[\epsilon\beta(t)]^2 + 2\epsilon \beta'(t) + 6\epsilon (\cot t)\beta(t),
\end{equation}
and, in particular,
\begin{equation}\label{ZepsilonTC}
    Z^\epsilon(t) = \frac{\epsilon}{2\sin^2t} [(3\epsilon + 12)\cos^2 t - 4]  
    >0
    \quad\mbox{for } t\in (0,t_c].
\end{equation}
Moreover, by \eqref{Zepsilon}, it is clear that there exists a constant $b_0>0$, depending only
on $\beta$, such that
\begin{equation}\label{ZepsilonLB}
	Z^\epsilon(t)\ge -\epsilon\, b_0 \quad \mbox{for } t\in (0,T].
\end{equation}

\subsection{Existence of a $\mu_\epsilon$-bubble}\

Let $S(r,g)$ (resp., $B(r,g)$) denote the geodesic sphere (resp., open geodesic ball) of radius $r$,  taken with respect to the metric $g$ and
centered at $\zerob\in \B_T$. An asymptotic expansion of the mean curvature function (see Lemma 3.4 of \cite{FST09}) gives:  for small $r>0$ and all $q\in S(r,g)$,
\begin{equation}\label{MCasymp1}
	H_{S(r,g)}(q) = \frac{2}{r} + O(r), \quad \hat H(q) = \frac{2}{t(q)} + O(t(q)).
\end{equation}
Since $g\ge \hat g$, we have $r \ge t(q)$; then by \eqref{muepsilon} and \eqref{betaDefn},
as long as $r<t_c$, we have
\begin{equation}\label{MCasymp2}
	\mu_\epsilon(t(q)) = \frac{2+\epsilon}{t(q)} + O(t(q)) \ge \frac{2+\epsilon}{r} + O(t(q)),\quad
	q\in S(r,g).
\end{equation}
It is now clear that 
there exists an $r_\epsilon <\epsilon$ such that $\mu_\epsilon > H_{S(r_\epsilon,g)}$ on 
$S(r_\epsilon,g)$. 
On the other hand, we have $H_g\ge 2\cot T > \mu_\epsilon(T)$ on $S_T$, where the first
inequality is part of \eqref{capMetricCond}, and the second inequality is due to the choice of
$\mu_\epsilon$ and $\beta$. 
Therefore, $\mu_\epsilon$ satisfies the 
barrier condition (see Definition~\ref{barrierDefn}) applied to the Riemannian band
$(M_\epsilon,g)$, where $M_\epsilon = \B_T\setdiff B(r_\epsilon,g)$, with the distinguished boundaries:
$\partial_- = S(r_\epsilon,g)$ and $\partial_+ = S_T$.
The lemma below follows directly from Lemma~\ref{mubbExReg}.

\begin{lemma}\label{MBepsilonExistenceLemma}
In the Riemannian band $(M_\epsilon,g; S(r_\epsilon,g), S_T)$ 
there exists a minimal $\mu_\epsilon$-bubble $\Omega_\epsilon$; moreover,
$\partial\Omega_\epsilon\setdiff S(r_\epsilon,g)$ is disjoint from $S_T$,
and it has
 a connected component $\Sigma^\epsilon_0$ whose projection onto $\S^2$ has nonzero degree.
\end{lemma}

\begin{lemma}\label{MeetingCptSet}
    $\Sigma^\epsilon_0\cap \B_{[t_c,T]}$ is nonempty.
\end{lemma}
\begin{proof}
    Otherwise, $Z^\epsilon > 0$ on 
    $\Sigma^\epsilon_0$, which contradicts Proposition~\ref{levelsetProp}(\ref{levelset_Z}).
\end{proof}

\subsection{A `no-crossing' property of $\Sigma^\epsilon_0$}\label{sec_noCross}\

From now on, let $t_* \in (0,t_c)$ be fixed. We will begin by assuming that
$\Sigma^\epsilon_0\cap\B_{t_*}$
were nonempty; consequences of this hypothesis will be developed progressively with three lemmas
(Lemmas~\ref{slopeLemma}, \ref{posilemmaWArea} and \ref{posiLemma}). 
Based on these lemmas, we prove
that $\Sigma^\epsilon_0$ must be disjoint  from $\B_{t_*}$ for small enough $\epsilon$ (Proposition~\ref{noCrossing}) .

In the following, let  $\hat\nu$ denote the outward-pointing
unit normal on $\Sigma^\epsilon_0$ with respect to $\hat g$, and 
let
$\Phi$ denote the projection map from $\Sigma^\epsilon_0$ 
to $\S^2$ (cf. \eqref{projPhiDef}).  

\begin{lemma}\label{slopeLemma}
If  $\Sigma^\epsilon_0\cap\B_{t_*}$ were nonempty, then there would exist a point $q\in \Sigma^\epsilon_0\cap \B_{[t_*, T]}$ such that the angle $\angle_{\hat g}(\hat \nu, \partial_t)  
\in [\alpha, \pi - \alpha]$ at $q$, where
\begin{equation}\label{alpha}
       \alpha = \arctan\left( \frac{t_c - t_*}{2\pi}\right)< \frac{\pi}{4}. 
    \end{equation}
\end{lemma}

\begin{proof}
We argue by contradiction, so let us assume that
$\angle_{\hat g}(\hat\nu,\partial_t)\in [0,\alpha)\cup(\pi - \alpha,\pi]$ everywhere
on $\Sigma^\epsilon_0\cap \B_{[t_*,T]}$.
Because  $\Sigma^\epsilon_0$ is connected and intersects both $S_{t_*}$ (by assumption)
and $S_{t_c}$ (by Lemma~\ref{MeetingCptSet}), the image of $t|_{\Sigma^\epsilon_0}$
contains the interval $[t_*, t_c]$.

Let $t'\in (t_*,t_c)$ be a regular value of $t|_{\Sigma^\epsilon_0}$
that is sufficiently close to $t_*$. Because $\Sigma^\epsilon_0$ is connected, 
there exists a connected component $\Ecal\subset\Sigma^\epsilon_0\cap \B_{(t', t_c]}$
whose closure $\bar \Ecal$ intersects both $S_{t'}$ and $S_{t_c}$. On $\Ecal$, the angle
$\angle_{\hat g}(\hat \nu,\partial_t)$ can only take value in one of the intervals $[0,\alpha)$
and $(\pi - \alpha, \pi]$, but not both. Without loss of generality, let us assume that 
$\angle_{\hat g}(\hat\nu,\partial_t)\in [0,\alpha)$ on $\Ecal$. 

Since $t'$ is a regular value of $t|_{\Sigma^\epsilon_0}$, $\bar\Ecal$ meets $S_{t'}$ transversely. 
In particular, $\mathscr{C}:=\bar\Ecal\cap S_{t'}$ is a disjoint union of finitely many circles. It is easy to see that $S_{t'}\setdiff \mathscr{C} = U_1\cup U_2$ for some open subsets
$U_i\subset S_{t'}$ with $\partial U_i = \mathscr{C}$ $(i = 1,2)$.

Both $U_i$ and $\Ecal$ are oriented, and the orientations are associated to the respective
normal directions, $\partial_t$ and $\hat \nu$, by the right-hand rule. 
The orientation on $\mathscr{C}$ induced by $\Ecal$ must completely agree with 
that induced by either $U_1$ or $U_2$; otherwise, gluing $\Ecal$ with either $U_1$ or $U_2$
along $\mathscr{C}$ and smoothing would yield a non-orientable closed surface embedded
in $\B_T$, which is impossible.

Thus, we can assume that $U_1$ and $\Ecal$ induce \emph{opposite}  orientations on $\mathscr{C}$.
Since $\angle_{\hat g}(\hat\nu, \partial_t) \in [0,\alpha)$ on $\Ecal$, it is easy to see that 
the restriction of $\Phi$ to $\bar\Ecal\cup U_1$ is a local homeomorphism to $\S^2$. 
Since $\bar\Ecal\cup U_1$ is compact, $\Phi|_{\bar\Ecal\cup U_1}$ is a covering map; this map
must be a homeomorphism, since $\S^2$ is simply connected and $\bar\Ecal\cup U_1$ is connected.

Pick any $x\in \Ecal\cap S_{t_c}$. Choose a shortest (regular) curve $\Gamma: [0,1]\rightarrow\Phi(\bar\Ecal)$ 
connecting $\Gamma(0) = \Phi(x)$ and $\partial(\Phi(\Ecal))$; in particular,
\begin{equation}\label{gammaLength}
	\length_{g_{\S^2}}(\Gamma) \le \pi. 
\end{equation}	

Now let $\gamma = (\Phi|_{\bar\Ecal})^{-1}\circ \Gamma$, and write its tangent vectors $\gamma'$ as the sum of $\gamma'_N$ (parallel to $\partial_t$) and $\gamma'_T$ (tangent to $t$-level sets). By $\hat g\le g_{\S^2}+\ed t^2$ and the hypothesis
$\angle_{\hat g}(\hat\nu,\partial_t)\in [0,\alpha)\cup (\pi - \alpha, \pi]$, 
we obtain the estimate
\begin{equation}\label{gammaNormalEst}
	 |\gamma'_N|_{\hat g} \le (\tan\alpha)|\gamma'_T|_{\hat g}
	 \le (\tan\alpha) |\ed\Phi(\gamma')|_{g_{\S^2}}.
\end{equation}
Hence,
    \begin{equation}\label{contra1}
    	t_c - t'\le 
        \int_{\gamma}
        |\gamma'_N|_{\hat g}
        \le (\tan \alpha)\cdot\length_{g_{\S^2}}(\Phi(\gamma))\le \pi \tan\alpha \le \frac{1}{2}\left(t_c - t_*\right),
    \end{equation}
  where the first inequality holds because $\gamma(0)\in S_{t_c}$ and $\gamma(1)\in S_{t'}$;
  the second and third inequalities are due to \eqref{gammaNormalEst} and \eqref{gammaLength},
  respectively; the last inequality holds by the choice of $\alpha$.
 Since $t'$ is close to $t_*$, \eqref{contra1} is a contradiction.
\end{proof}

\begin{cor}\label{slopeCor}
   In Lemma~\ref{slopeLemma} we can choose $q$ such that: $\angle_{\hat g}(\hat \nu, \partial_t) = \alpha$ or $\pi - \alpha$ at $q$. 
\end{cor}

\begin{proof}
    In $\Sigma^\epsilon_0$ there exists a point at which $t$ attains global maximum. At that point $\hat \nu = \pm \partial_t$. Thus, by continuity of angle, there exists a point $q\in \Sigma^\epsilon_0\cap \B_{t_*}$ at which the angle between $\hat \nu$ and $\partial_t$ is equal to either $\alpha$ or $\pi - \alpha$.
\end{proof}

\begin{lemma}\label{posilemmaWArea}
Let $\alpha$ be defined by \eqref{alpha}. If  $\Sigma^\epsilon_0\cap\B_{t_*}$ were nonempty, then there would exist a constant $S = S(g,\hat g, \beta, t_*)>0$, independent of $\epsilon$,
 and an open subset $U_\epsilon\subset \Sigma_0^\epsilon\cap \B_{[t_*/2, T]}$ such that 
    \begin{enumerate}[\rm (1)]
    \item at each point $q\in U_\epsilon$, $\angle_{\hat g}(\hat\nu, \partial_t)\in (\alpha/2, 2\alpha)\cup (\pi - 2\alpha, \pi - \alpha/2)$;
    \item $\displaystyle\int_{U_\epsilon}|\Phi^* d\sigma_{\S^2}|\ge S$.
   \end{enumerate}
\end{lemma}

\begin{proof}
    To begin with, let $q$ be as in Corollary~\ref{slopeCor}. For any unit tangent vector $X$ (with respect to $\hat g$) of $\Sigma^\epsilon_0$, we have
    \begin{equation}\label{angleturningRate}
        |X\<\hat \nu, \partial_t\>_{\hat g}|
        = |\<\hat\nabla_X \hat\nu, \partial_t\>_{\hat g} + \<\hat\nu, \hat\nabla_X \partial_t\>_{\hat g}|
        \le |\hat \II|_{\hat g} + |\hat\nabla\partial_t|_{\hat g}.
    \end{equation}
    where $\hat \nabla$ is the connection of $\hat g$.
    It is clear that there exists a constant $C = C(\hat g, t_*)$ such that
    $|\hat\nabla\partial_t|_{\hat g}\le C$ on $\B_{[t_*/2,T]}$. 
    Moreover, by applying Lemma~\ref{sffLemma} and by comparing 
    between $|\II|_{g}$ and $|\hat \II|_{\hat g}$, it is not difficult to see that 
    there exists a constant $C' = C'(g,\hat g,\beta,  t_*)$ such that 
    $|\hat\II|_{\hat g}\le C'$ on $\Sigma^\epsilon_0\cap\B_{[t_*/2, T]}$
    for all sufficiently small $\epsilon$.
        Thus, there exists a constant $\rho = \rho(g,\hat g,\beta, t_*)>0$ such that on the geodesic ball 
    \[
    U_\epsilon:=\left\{x\in \Sigma^\epsilon_0: \dist_{\hat g_{\Sigma^\epsilon_0}}(x,q)\le\rho\right\}
    \]
    we have \begin{equation}\label{slopeSqueeze}
        \angle_{\hat g}(\hat \nu, \partial_t) \in \left(\frac{\alpha}{2}, 2\alpha\right)
        							\cup \left(\pi -  2\alpha, \pi - \frac{\alpha}{2}\right).
    \end{equation}
     It is easy to see 
     that $\Phi(U_\epsilon)$
     contains a ball $\mathscr{B}$ of radius $\cos(2\alpha)\rho$ in $\S^2$. The proof is 
     complete by taking  $ S:= \area_{g_{\S^2}}(\mathscr{B})$.
\end{proof}

\begin{lemma} \label{posiLemma}
If $\Sigma^\epsilon_0\cap \B_{t_*}$ were nonempty, then we would have
\begin{equation}\label{posiDiff}
    \int_{\Sigma^\epsilon_0} \frac{2}{\varphi^2} d\sigma_{\hat g}
       - 2\int_{\Sigma^\epsilon_0} |\Phi^*d\sigma_{\S^2}| \ge A_0
\end{equation}
for some positive constant 
$A_0$ that is independent of $\epsilon$.
\end{lemma}

\begin{proof}
Up to sign, the area form $d\sigma_{\hat g}$ induced by $\hat g$ on each 
tangent space of $\Sigma^\epsilon_0$ is equal to 
\[
	\frac{1}{\cos(\angle_{\hat g}(\hat \nu,\partial_t))} \varphi^2\Phi^*d\sigma_{\S^2}
\]
provided that $\hat\nu$ is not orthogonal to $\partial_t$.
Thus, by Lemma~\ref{posilemmaWArea},  we have
\begin{equation}\label{UepsilonIntEst}
\begin{alignedat}{1}
    \int_{U_\epsilon}\frac{2}{\varphi^2}d\sigma_{\hat g} &\ge
    \int_{U_\epsilon} \frac{2}{\varphi^2} \frac{1}{\cos(\alpha/2)}\varphi^2|\Phi^*d\sigma_{\S^2}|
    \\
    &\ge 2S\left(\frac{1}{\cos(\alpha/2)} - 1\right) + 2\int_{U_\epsilon}|\Phi^*d\sigma_{\S^2}|.
\end{alignedat}    
\end{equation}
On the other hand, by Lemma~\ref{areaCompLemma},
\begin{equation}\label{UepsilonCplmEst}
	\int_{\Sigma^\epsilon_0\setdiff U_\epsilon}\frac{2}{\varphi^2}d\sigma_{\hat g}
	\ge 2\int_{\Sigma^\epsilon_0\setdiff U_\epsilon}|\Phi^*d\sigma_{\S^2}|.
\end{equation}
Adding \eqref{UepsilonIntEst} with \eqref{UepsilonCplmEst} and rearranging terms, 
we get
\begin{equation}\label{finalAug}
     \int_{\Sigma^\epsilon_0} \frac{2}{\varphi^2} d\sigma_{\hat g}
       - 2\int_{\Sigma^\epsilon_0} |\Phi^*d\sigma_{\S^2}|
       \ge 2S\left(\frac{1}{\cos(\alpha/2)} - 1\right).
\end{equation}
The proof is complete by taking $A_0$ to be the RHS of \eqref{finalAug}. 
\end{proof}

\begin{prop}\label{noCrossing}
    For sufficiently small $\epsilon$,  
     $\Sigma^\epsilon_0$ must be disjoint from the set $\B_{t_*}\subset \B_T$.  
\end{prop}

\begin{proof}
By \eqref{ZepsilonLB} and the proof of Lemma~\ref{RplusLemma},
we obtain
\begin{equation}\label{RpluswithDeduction}
    R_+^{\mu_\epsilon} 
    \ge \frac{2}{\varphi^2} - 2b_0\epsilon   \qquad\mbox{on } \Sigma^\epsilon_0.
\end{equation}
For small $\epsilon$, Remark~\ref{ZepsilonRelaxRmk} and
Lemma~\ref{sphereLemma} imply that $\Sigma^\epsilon_0$ is homeomorphic to $\S^2$.
Moreover, 
since $\Omega_\epsilon$ is a $\mu_\epsilon$-bubble, the area of $\Sigma^\epsilon_0$
with respect to $g$ has an upper bound $C_0>0$, which can be chosen to depend only on the metric $g$ and not on $\epsilon$. 

Now suppose that $\Sigma_0^\epsilon\cap \B_{t_*}\ne \emptyset$. Then from \eqref{RpluswithDeduction}, 
\eqref{sigmaVSsigma} and \eqref{posiDiff}, we obtain:
\begin{equation}\label{augOfInt}
    \int_{\Sigma^\epsilon_0} R_+^{\mu_\epsilon} d\sigma_g
    \ge \int_{\Sigma^\epsilon_0} \frac{2}{\varphi^2}d\sigma_{\hat g} - 2\epsilon b_0C_0
     \ge \left( A_0 - 2\epsilon b_0 C_0\right)	
       	+ 2\int_{\Sigma^\epsilon_0} |\Phi^*d\sigma_{\S^2}|.
\end{equation}
For small enough $\epsilon$, $A_0>2\epsilon b_0 C_0$;
by stability of $\Sigma^\epsilon_0$, the analogue  of \eqref{GaussBonnet}  reads
\begin{equation}\label{noCross_GB}
	4\pi\chi(\S^2) =  \int_{\Sigma_0^\epsilon} R_{\Sigma_0^\epsilon} d\sigma_g 
	\ge \int_{\Sigma_0^\epsilon} R_+^{\mu_\epsilon} d\sigma_g 
	> 2\int_{\Sigma^\epsilon_0} |\Phi^*d\sigma_{\S^2}| \ge 8\pi;
\end{equation}
a contradiction.
\end{proof}

\begin{remark}\label{Rmk_altApr2Ineq}
 There is another way to get \eqref{noCross_GB}, which does not 
 rely on the assumption of an upper bound $C_0$ of $\area_g(\Sigma^\epsilon_0)$ but does
 rely on the fact that $\varphi\le 1$. In fact, \eqref{RpluswithDeduction} implies
 that $R^{\mu_\epsilon}_+ \ge 2\varphi^{-2}(1 - b_0\epsilon)$, and again by \eqref{sigmaVSsigma}, \eqref{posiDiff} and the degree assumption 
 we have
 \[
 	\int_{\Sigma^\epsilon_0} R^{\mu_\epsilon}_+d\sigma_g \ge 
	(1-b_0\epsilon) \left(A_0 + 2\int_{\Sigma^\epsilon_0}|\Phi^*d\sigma_{\S^2}|\right)
	\ge (1-b_0\epsilon) (A_0 + 8\pi) > 8\pi
 \] 
 for small enough $\epsilon$. 
\end{remark}

\subsection{Existence of a minimizer}\label{limitingSec}\

Let $M_\epsilon$, $\Omega_\epsilon$ and $\Sigma^\epsilon_0$ be as in Lemma~\ref{MBepsilonExistenceLemma}. 
We now study how $\Omega_\epsilon$ and $\Sigma^\epsilon_0$ behave
as $\epsilon\rightarrow 0$.

Recall from \eqref{tcDef} the definition of $t_c$, and let
$t_*\in(0,t_c)$ be fixed. By considering small enough $\epsilon$, we can assume
 $\Sigma^\epsilon_0$ to be homeomorphic to $\S^2$ and disjoint from $\B_{t_*}$.
 
For a fixed $\epsilon$, since $\Sigma^\epsilon_0$ is disjoint from $S_T$, the Jordan--Brouwer separation theorem applies. As a result,  
$\B_T\setdiff\Sigma^\epsilon_0$ has exactly two connected components, say  $\Ucal^\epsilon_-$ and $\Ucal^\epsilon_+$.
Without loss of generality, let us assume that 
$\nu$ points away from $\Ucal^\epsilon_-$ along $\Sigma^\epsilon_0$. 
Given any constant $\delta >0$, let us define
\begin{equation}\label{paddingDef}
\begin{alignedat}{1}
	W^\epsilon_{-\delta} &:= 
			\left\{x\in \Ucal^\epsilon_-: \dist_{g}(x,\Sigma^\epsilon_0)\le \delta\right\},
			\\
	W^\epsilon_{+\delta} &:= 
			\left\{x\in \Ucal^\epsilon_+: \dist_{g}(x,\Sigma^\epsilon_0)\le \delta\right\},
\end{alignedat}
\end{equation}
where distance is taken in $(\B_T,g)$.

\begin{lemma}\label{paddingLemma}
	There exists a constant $\delta>0$, independent of $\epsilon$, such that for all small enough $\epsilon$
	we have $W^\epsilon_{-\delta}\subset\mathring\Omega_{\epsilon}$
	and $W^\epsilon_{+\delta}\cap\Omega_{\epsilon} = \emptyset$.
\end{lemma}
\begin{proof}
	Since in $\B_{[t_*/2, T]}$ 
	all derivatives of $\mu_\epsilon$ are uniformly bounded, it follows from Lemma~\ref{sffLemma} 
	that the norm of the second fundamental form of $\partial\Omega_\epsilon\cap \B_{[t_*/2,T]}$
	is also uniformly bounded. If some other component $\Sigma'$ in 
	$\partial\Omega_\epsilon$
	were to get arbitrarily close to 
	$\Sigma^\epsilon_0$, then a suitable surgery (i.e., a connected sum of $\Sigma^\epsilon_0$
	and $\Sigma'$ performed within $M_\epsilon$)
	would yield a Caccioppoli set that has
	has strictly less brane action, contradicting the minimality of $\Omega_\epsilon$.
\end{proof}

Now we fix a sequence $\{\epsilon_i\}\rightarrow 0$ and corresponding
sequences  of $\Omega_{\epsilon_i}$ and $\Sigma^{\epsilon_i}_0$.

\begin{lemma}\label{convCC}
	The sequence $\{\Omega_{\epsilon_i}\}$ subconverges to a Caccioppoli set 
	$\Omega\subset \B_T$ where convergence is interpreted via the characteristic functions
	with respect to the $L^1_{\loc}$-norm. Moreover, 
	\begin{enumerate}[\rm(1)]
	\item $\partial\Omega\setdiff\{\zerob\}$ is smooth and embedded, and
	\item  $\Omega$ is a minimizer in the sense that
		$\Acal^{\hat H}_{\Omega}(\Omega')\ge 0$ for any Caccioppoli set $\Omega'$ 
		with $\Omega'\Delta\Omega \Subset \B_T\setdiff\{\zerob\}$.
	\end{enumerate}	
\end{lemma}
\begin{proof}
	The existence of a convergent subsequence and that of $\Omega$ follow from
	standard theory of BV functions (cf. \cite[Theorem 1.20]{Giusti84}), and let us replace
	$\{\Omega_{\epsilon_i}\}$ by that subsequence.
	 
	Now let $K\subset \B_T\setdiff\{\zerob\}$ be any compact domain.
	For sufficiently large $i$, 
	the second fundamental form of $\partial\Omega_{\epsilon_i}\cap K$
	has a uniform upper bound, and thus $\partial\Omega_{\epsilon_i}\cap K$
	subconverges to a smooth hypersurface $\Scal\subset K$ in the graph sense. By using
	Lemma~\ref{paddingLemma}, it is easy to see that $\Scal$ is embedded and  $\Scal = \partial\Omega\cap K$.
	Since $K$ is arbitrary, we conclude (1).
	
	To show that $\Omega$ is a minimizer,  
	we argue by contradiction. 
	Suppose that there exists a Caccioppoli set $\Omega'$
	and a constant $c>0$ such that $\Omega'\Delta\Omega\Subset\B_T\setdiff\{\zerob\}$ and 
	$\Acal^{\hat H}_\Omega(\Omega')\le -c<0$. 
	Let us choose a compact domain  $K\subset\B_T\setdiff\{\zerob\}$ with smooth boundary
	 such that $\Omega'\Delta\Omega\Subset \mathring K$.
	Consider a thin tubular neighborhood
	$\Tcal$ of $\partial\Omega\cap K$ that is generated by the unit normal field along 
	$\partial\Omega\cap K$; as $\Tcal$ is diffeomorphic to 
	$(\partial\Omega\cap K)\times I$ for some interval $I$, 
	we may modify $K$ such that the image of
	$(\partial\Omega\cap \partial K)\times I$ is equal to
	$\partial\Tcal\cap \partial K$ (in particular, $\partial\Omega$ is transversal
	to $\partial K$). 
	Note that for large $i$, $S(r_{\epsilon_i}, g)$ would be disjoint from $K$, and
	$\partial\Omega_{\epsilon_i}\cap K$ would be completely contained in $\Tcal$.
	%figure
	\begin{figure}[h!]
		\begin{subfigure}{0.4\textwidth}
		\centering
		\includegraphics[scale = 0.32]{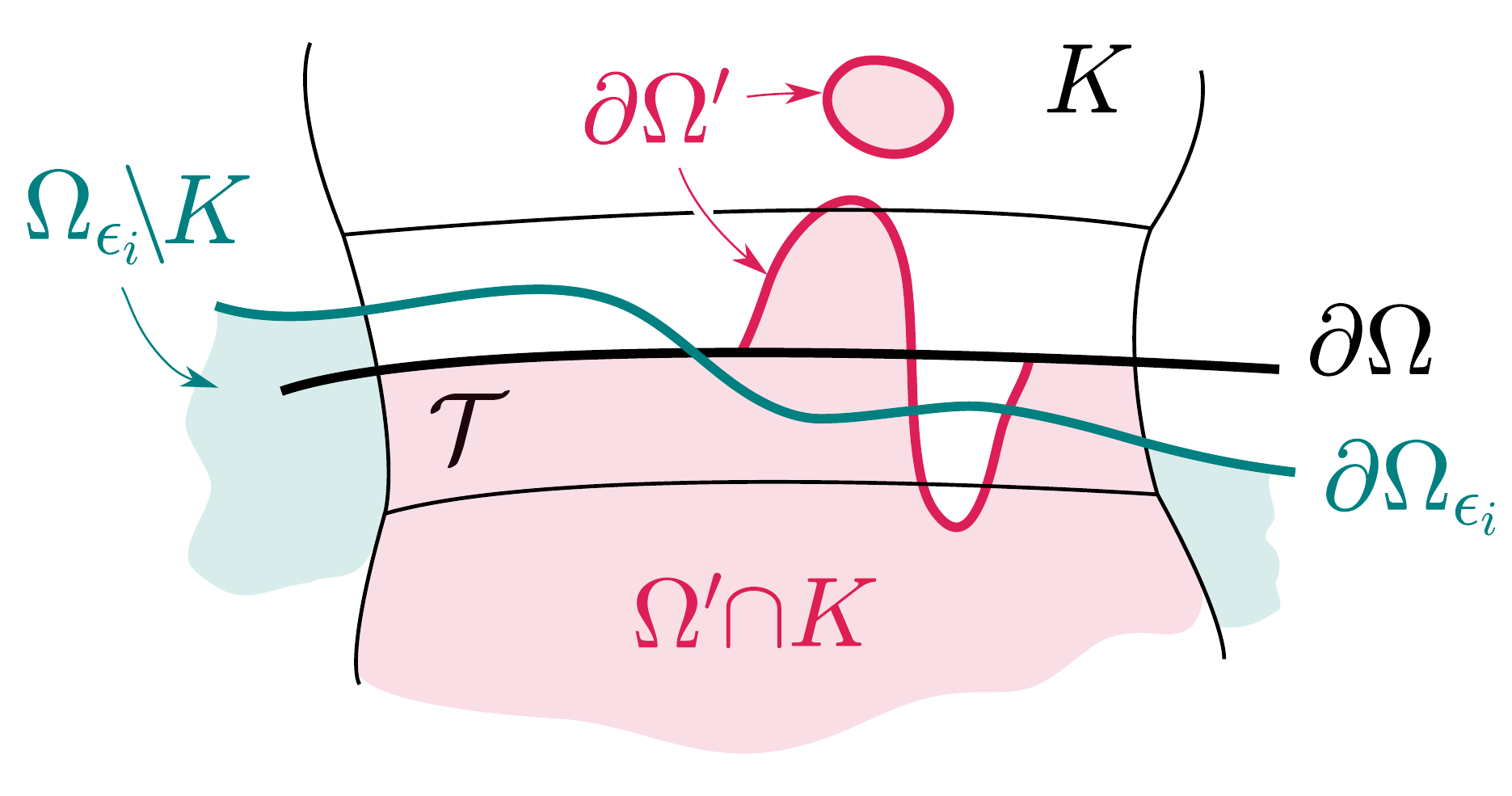}
		\end{subfigure}
		\qquad
		\begin{subfigure}{0.4\textwidth}
		\centering
		\includegraphics[scale = 0.32]{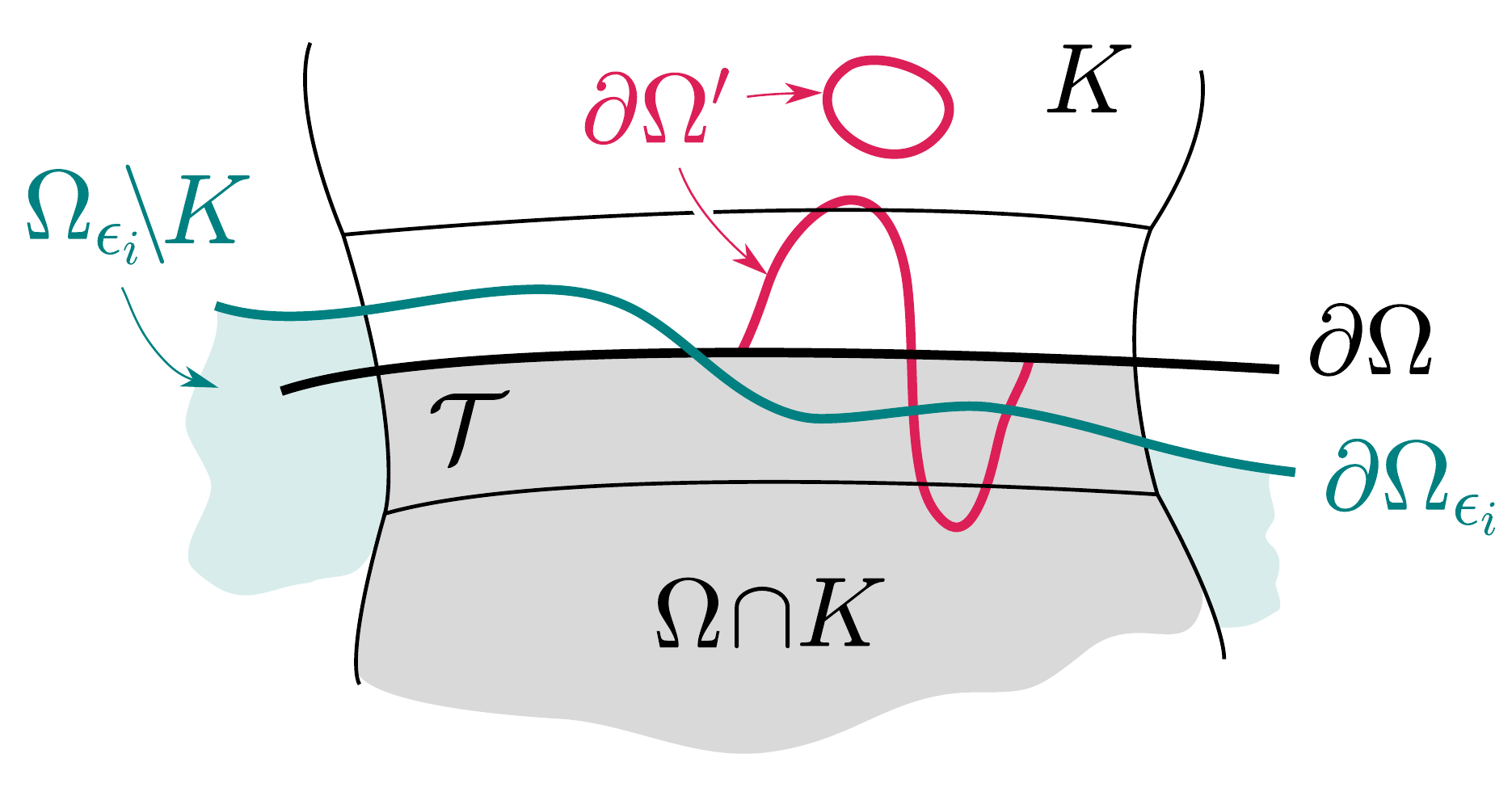}
		\end{subfigure}
		\caption{The shaded regions represent $\Omega_{\epsilon_i}'$ (left figure) and
					 $\Omega^*_{\epsilon_i}$ (right figure).}
		\label{Fig_Omega}
	\end{figure}
	
	Now consider the following Caccioppoli sets (see Figure~\ref{Fig_Omega})
	\begin{equation}
		\Omega_{\epsilon_i}' :=(\Omega_{\epsilon_i}\setdiff K)\cup (\Omega' \cap K),
		\qquad
		\Omega_{\epsilon_i}^* := (\Omega_{\epsilon_i}\setdiff K) \cup (\Omega\cap K).
	\end{equation}
	We claim that, for sufficiently large $i$,
	\begin{equation}\label{OeOestar}
		\Acal^{\mu_{\epsilon_i}}_{\Omega_{\epsilon_i}}(\Omega_{\epsilon_i}^*)\le \frac{c}{4}.
	\end{equation}
	To see this, note that
	$\chi_{\Omega_{\epsilon_i}^*} - \chi_{\Omega_{\epsilon_i}}$ is just
	$\chi_{\Omega_{\epsilon_i}\cap K} - \chi_{\Omega\cap K}$; since $\mu_{\epsilon_i}|_{K}$
	is uniformly bounded and $\chi_{\Omega_{\epsilon_i}}\rightarrow \chi_{\Omega}$ in 
	$L^1$, we have
	\begin{equation}\label{OeOestar_muest}
		\int_{\B_T} (\chi_{\Omega_{\epsilon_i}^*} - \chi_{\Omega_{\epsilon_i}})\mu_{\epsilon_i}
		\rightarrow  0 \qquad (i\rightarrow\infty).
	\end{equation}
	Moreover, it is easy to see that
	\begin{equation}
		\Hcal^2(\partial\Omega_{\epsilon_i}^*) - \Hcal^2(\partial\Omega_{\epsilon_i})
		\le \left[\Hcal^2(\partial\Omega\cap K) - \Hcal^2(\partial\Omega_{\epsilon_i}\cap K)\right]
			+ \Hcal^2(\partial\Tcal \cap \partial K).
	\end{equation}
	Thus, by graph convergence of $\partial\Omega_{\epsilon_i}\cap K$, we can choose 
	$\Tcal$ and $i$ such that 
	\begin{equation}\label{OeOestar_areaEst}
		\Hcal^2(\partial\Omega_{\epsilon_i}^*) - \Hcal^2(\partial\Omega_{\epsilon_i})
		\le \frac{c}{8}.
	\end{equation}
	On combining \eqref{OeOestar_muest} and \eqref{OeOestar_areaEst}, we obtain 
	\eqref{OeOestar} for large $i$.
	
	Now, since $\mu_{\epsilon_i}\rightarrow \mu$ in $L^1(K)$ and 
	$\Omega_{\epsilon_i}'\Delta \Omega_{\epsilon_i}^* = \Omega\Delta \Omega'\Subset\mathring K$, we have, for sufficiently large $i$,
	\begin{equation}\label{OePrimeOestar}
		\Acal^{\mu_{\epsilon_i}}_{\Omega_{\epsilon_i}^*}(\Omega_{\epsilon_i}') \le -\frac{c}{2}.
	\end{equation}
	On comparing \eqref{OeOestar} and \eqref{OePrimeOestar}, 
	we get 
	$\Acal^{\mu_{\epsilon_i}}_{\Omega_{\epsilon_i}}(\Omega_{\epsilon_i}') \le -c/4<0$,
	contradicting the minimality of $\Omega_{\epsilon_i}$. This proves (2).
\end{proof}

\begin{lemma}\label{convHS}
	Let $\Omega$ be as in Lemma~\ref{convCC}.
	The sequence $\{\Sigma^{\epsilon_i}_0\}$ subconverges to a smooth,
	closed stable $\hat H$-hypersurface $\Sigma_0\subset \B_{[t_*, T]}$, which is a $t$-level set in $\B_T$;
	moreover, $\Sigma_0\subset\partial\Omega$ and 
	$\partial\Omega\setdiff\Sigma_0 \Subset \B_T\setdiff\Sigma_0$.
\end{lemma}
\begin{proof}
	By our choice of $\{\epsilon_i\}$, 
	 all $\Sigma^{\epsilon_i}_0$ are contained in the compact set $\B_{[t_*, T]}$ and have 
	 a uniform upper
	 bound on their second fundamental form.
	Thus, by standard minimal surface theory (cf. \cite[Proposition 7.14]{CM11}), $\{\Sigma^{\epsilon_i}_0\}$ 
	subconverges to a smooth closed hypersurface $\Sigma_0$ 
	whose projection onto $\S^2$ has 
	nonzero degree. 
	Now recall that each $\Sigma^{\epsilon_i}_0$ is a stable $\mu_{\epsilon_i}$-hypersurface.
	Since all derivatives of $\mu_{\epsilon_i}$ respectively and uniformly 
	converge to those $\hat H$,
	$\Sigma_0$ is a stable $\hat H$-hypersurface;
	hence, $\Sigma_0$ is a $t$-level set, by Proposition~\ref{levelsetProp}. 
	
	To see that $\Sigma_0\subset\partial\Omega$, first suppose that $\Sigma_0\ne S_T$;
	in this case, it suffices to show that each open neighborhood of any $x\in \Sigma_0$
	must intersect both $\mathring\Omega$ and $\B_T\setdiff\Omega$, and this can be
	easily deduced from Lemma~\ref{paddingLemma}. 
	The case of $\Sigma_0 = S_T$ is similar. Also by Lemma~\ref{paddingLemma},
	$\Sigma_0$ has a tubular neighborhood that is disjoint from all other components of 
	$\partial\Omega$, hence 	$\partial\Omega\setdiff\Sigma_0 \Subset \B_T\setdiff\Sigma_0$.

\end{proof}

On combining Lemmas~\ref{convCC} and \ref{convHS}, we immediately get the following.

\begin{prop}\label{initProp}
	Let $g$ be a Riemannian metric on $\B_T$ satisfying \eqref{capMetricCond}. Then 
	there exists a Caccioppoli set $\Omega\subset \B_T$ and  a connected
	component $\Sigma_0\subset \partial\Omega$ that satisfy 
	Assumption~\ref{relativeCCassum}.
\end{prop}

Theorem~\ref{capRigThm} follows directly from Propositions~\ref{rigidityProp} and \ref{initProp}.

% Generalizations

\section{Generalizations}\label{Sec_gen}

In this section we discuss a few variants of Theorem \ref{capRigThm}. 

To begin with, we consider a version of Gromov's rigidity theorem for the 
doubly punctured sphere (see \cite[Sections 5.5 and 5.7]{Gromov21Four}), restricted to the $3$-dimensional case.
 
\begin{theorem}\label{capRigThm2}
 Let $(\S^3\setdiff\{O,O'\},\hat g)$ be the standard $3$-sphere with a pair of antipodal points removed, and let $h\ge 1$ be a smooth function on $\S^3\setdiff\{O,O'\}$.
 Suppose that $g$ is another Riemannian metric on $\S^3\setdiff\{O,O'\}$ satisfying
\begin{equation}\label{noncompactCond_txt}
 	 g\ge h^4 \hat g\quad \mbox{and}\quad R_g\ge h^{-2}  R_{\hat g}.
\end{equation}
 Then $h\equiv 1$, and $g = \hat g$.
\end{theorem}
\begin{proof}
For convenience, let us use slightly different notations than those introduced at the beginning of Section~\ref{rigSection} by representing $\S^3\setdiff\{O,O'\}$ as $\Bbb_{(-\pi/2,\pi/2)}\cong \S^2\times(-\pi/2,\pi/2)$ with $t$ 
being the coordinate on $(-\pi/2,\pi/2)$. Under this representation we have $\varphi(t) = \cos t$ and
\begin{equation}
	\hat H(t) = - 2\tan t
\end{equation} 
instead of \eqref{HhatDef}.
Now for $\alpha\in (0,\pi/2)$ sufficiently close to $\pi/2$, 
consider the Riemannian band $\Bscr_{\alpha}:=(\Bbb_{[-\alpha,\alpha]},g; S_{-\alpha}, S_\alpha)$ and the functions 
\begin{equation}
	t_{\alpha}=\frac{t}{\alpha}\cdot \frac{\pi}{2}
	\qquad\mbox{and}\qquad 
	\mu_{\alpha}= -2\tan t_\alpha\text{ on } \Bbb_{(-\alpha,\alpha)},
\end{equation}
and consider the problem of finding ${\mu}_{\alpha}$-bubbles in ${\Bscr}_{\alpha}$. Since $\mu_{\alpha}\rightarrow \pm\infty$ as $t\rightarrow \mp\alpha$, ${\mu}_{\alpha}$ satisfies the barrier condition; thus, there exists a $\mu_\alpha$-bubble 
$\Omega_{\alpha}\subset \Bscr_\alpha$, which satisfies analogous properties as described in
Lemma~\ref{MBepsilonExistenceLemma}. 
Let $\Sigma^\alpha_0$ be a connected component of 
$\partial\Omega_{\alpha}\setdiff S_{-\alpha}$ whose projection to $\S^2$ has nonzero degree;
$\Sigma^\alpha_0$ is a stable ${\mu}_{\alpha}$-hypersurface, on which
\begin{equation}
\begin{alignedat}{1}
R_+^{\mu_\alpha} &=R_g + \frac{3}{2}({\mu}_{\alpha})^2 - 2|\ed{\mu}_{\alpha}|_g\\
&\geq \frac{1}{h^2}\left(\frac{R_{\S^2}}{\varphi^2}-\frac{3}{2}\hat{H}^2+2|\ed\hat{H}|_{\hat{g}}\right) +\frac{3}{2}(\mu_{\alpha})^2 - \frac{2}{h^2}|\ed\mu_{\alpha}|_{\hat{g}}\\
&\geq \frac{1}{h^2}\left(\frac{R_{\S^2}}{\varphi^2}+ Z_{\mu_\alpha}\right)
\end{alignedat}
\end{equation}
where the last step follows from the assumption $h\ge 1$ and the definition
\[
	Z_{\mu_\alpha}:= \frac{3}{2}(\mu_{\alpha}^2-\hat{H}^2)+2(\partial_{t}\mu_{\alpha}-\partial_{t}\hat{H}). 
\]
By a careful estimate of $Z_{\mu_\alpha}$ using the mean value theorem, it is not difficult to show that there exists a constant $t_{c}>0$ such that 
\begin{equation}\label{twoestimates}
Z_{\mu_\alpha}>0\quad \text{for } t\in(-\alpha,-t_{c})\cup(t_{c},\alpha), 
\quad\text{and}
\quad \varphi^{2}Z_{\mu_\alpha}\geq C(\alpha)\quad\text{for } t\in (-\alpha,\alpha).
\end{equation}
where $C(\alpha)<0$ is a constant depending only on $\alpha$ and satisfies $C(\alpha)\rightarrow 0$ as $\alpha\rightarrow \pi/2$. 
Similar to the proof of Proposition \ref{noCrossing}, here \eqref{twoestimates} implies that $\Sigma^\alpha_0$
is contained in a fixed compact domain in $\Bbb_{(-\pi/2,\pi/2)}$ that is independent
of the choice of $\alpha$. Thus, as $\alpha\rightarrow \pi/2$, such $\Sigma^\alpha_0$'s
subconverge to a stable $\hat H$-hypersurface, and an analogue of Proposition~\ref{initProp} can be obtained.
An analogue of Proposition~\ref{levelsetProp} and a foliation argument yield that $h \equiv 1$ and $g = \hat g$.
\end{proof}

\begin{remark}\label{confRmk}
 The assumption $h\ge 1$ is important for Theorem \ref{capRigThm2} to hold.  Without this assumption, one may let  $g=\cos^{2}{t}(\ed t^{2}+g_{\S^2})\ne\hat g$ on $\S^{3}\setdiff\{O,O'\}\cong \S^2\times(-\pi/2,\pi/2)$ and take $h=(\cos{t})^{1/2}$, and it
 is easy to check that \eqref{noncompactCond_txt} is satisfied---in particular, $R_g = (2+4\cos^2t)(\cos t)^{-4}$ and $h^{-2}R_{\hat g} = 6(\cos t)^{-1}$, so $R_{g}\geq h^{-2}R_{\hat g}$. 
 \end{remark}

Theorem~\ref{capRigThm} has Euclidean and hyperbolic analogues.
Putting together, let us take
\begin{equation}\label{k-capMetricStd}
	\hat g_{\kappa} =  \varphi_{\kappa}(t)^2 g_{\S^2} + \ed t^2 \qquad \text{on } \B_{T}
\end{equation}
where 
\[
	\varphi_{\kappa}(t) =
	\left\{
		\begin{alignedat}{2}
			& \quad\sin \sqrt{\kappa}t&&,\quad \kappa>0, \\
			 &\quad\qquad t &&,\quad \kappa = 0,\\
			 & \sinh \sqrt{-\kappa}t &&, \quad\kappa<0,
		\end{alignedat}
	\right.
\]
and $T\in (0,\pi/\sqrt{\kappa})$ if $\kappa >0$; $T>0$ if $\kappa\le 0$.
In particular, $\sec(\hat{g}_{\kappa}) = \kappa$, and $\hat{H}_{\kappa}(t)={2\varphi_{\kappa}'(t)}/{\varphi_{\kappa}(t)}$.

\begin{theorem}\label{capRigThm1_conf}
Let $\B_T, \hat g_\kappa$ be as above.
Let $g$ be a Riemannian metric on $\B_T$ satisfying
\[
	g\geq h^{4}\hat{g}_\kappa,\quad 
	R_{g}\geq h^{-2}R_{\hat{g}_\kappa},\quad 
	H_{\partial \B_T}\geq \hat{H}_{\kappa}(T),
\]
for some smooth function $h \geq 1$ defined on $\B_T$. 
Then $h\equiv1$,  and $g = \hat g_{\kappa}$.
\end{theorem}

As pointed out by Gromov in \cite[Section 5.5]{Gromov21Four}, a key fact that allows the different cases (corresponding to different choices of $\kappa$) in 
Theorem~\ref{capRigThm1_conf} to be treated similarly is that the function $\varphi_\kappa(t)$
is ``log-concave"---in other words, $\hat H_\kappa(t)$ is strictly decreasing in $t$ (cf. Lemma~\ref{RplusLemma} and Proposition~\ref{levelsetProp}). Having this in mind, the proof proceeds as that of
either Theorem~\ref{capRigThm} or Theorem~\ref{capRigThm2}, and we leave the details to the interested reader.

\begin{remark}
When $\kappa\leq 0$ and $T=+\infty$, whether Theorem \ref{capRigThm1_conf} holds remains unknown to us.
\end{remark}

%%%
%%%
%%%
%%%

%References

\bibliographystyle{alpha}

\end{document}